\def\arxivversion{}   
\ifdefined\arxivversion
  \documentclass{amsart}
\else
  \documentclass{mcom-l}
\fi

\ifdefined\arxivversion
    \usepackage[letterpaper,margin=1.25in,includefoot]{geometry}
\fi

\usepackage{amssymb}
\usepackage{graphicx}

\usepackage{algorithm}
\usepackage{algpseudocodex}
\usepackage{multirow}
\usepackage{xcolor}
\usepackage[numbers]{natbib}
\usepackage{url}
\usepackage{aliascnt}
\usepackage{mathtools}
\ifdefined\arxivversion
  \definecolor{linkblue}{RGB}{0,84,147}
  \definecolor{citegreen}{RGB}{0,122,94}
  \definecolor{urlmaroon}{RGB}{128,0,0}
  \usepackage{hyperref}
  \hypersetup{
    colorlinks = true,
    linkcolor  = linkblue,
    citecolor  = citegreen,
    urlcolor   = urlmaroon,
    pdftitle   = {Optimal Monte Carlo integration on matrix manifolds},
    pdfauthor  = {Aleksei G. Sorokin, Yihan Guo, Lek-Heng Lim},
  }
\fi
\usepackage[capitalise,noabbrev,nameinlink]{cleveref}  

\newcommand{\calO}{{\mathcal{O}}}
\newcommand{\bbR}{{\mathbb{R}}}
\newcommand{\bbC}{{\mathbb{C}}}
\newcommand{\bbH}{{\mathbb{H}}}
\newcommand{\bbN}{{\mathbb{N}}}
\newcommand{\bbF}{{\mathbb{F}}}
\newcommand{\bbE}{{\mathbb{E}}}
\newcommand{\calM}{{\mathcal{M}}}
\newcommand{\calG}{{\mathcal{G}}}
\newcommand{\calB}{{\mathcal{B}}}
\newcommand{\calZ}{{\mathcal{Z}}}
\newcommand{\tcalZ}{{\widetilde{\mathcal{Z}}}}
\newcommand{\calX}{{\mathcal{X}}}
\DeclareMathOperator{\diag}{diag}
\DeclareMathOperator{\trace}{trace}
\newcommand{\llvert}{\left\lvert}
\newcommand{\rrvert}{\right\rvert}
\newcommand{\llVert}{\left\lVert}
\newcommand{\rrVert}{\right\rVert}
\newcommand{\llangle}{\left\langle}
\newcommand{\rrangle}{\right\rangle}
\newcommand{\simiid}{\overset{\mathrm{IID}}{\sim}}
\DeclareMathOperator*{\argmin}{argmin}

\newcommand{\tw}{{\widetilde{w}}}
\newcommand{\tK}{{\widetilde{K}}}
\newcommand{\hw}{{\widehat{w}}}
\newcommand{\tgamma}{{\widetilde{\gamma}}}
\newcommand{\hf}{{\widehat{f}}}

\newcommand{\qqtqq}[1]{\qquad\mathrm{#1}\qquad}

\newcommand{\WCEm}[4]{\mathrm{WCE}_{{#1}}\left(#2,#3,#4\right)}

\newcommand{\h}{{\scriptscriptstyle\mathsf{H}}}

\newcommand{\s}{{\scriptscriptstyle\mathsf{S}}}
\newcommand{\la}{{\scriptscriptstyle\mathsf{L}}}
\let\O\undefined
\DeclareMathOperator{\O}{O}
\DeclareMathOperator{\U}{U}
\DeclareMathOperator{\V}{V}
\DeclareMathOperator{\Sp}{Sp}

\DeclareMathOperator{\Gr}{Gr}
\DeclareMathOperator{\LGr}{LGr}

\DeclareMathOperator{\spn}{span}

\DeclareMathOperator{\Flag}{Flag}

\DeclareMathOperator{\unif}{\textsc{Unif}}

\newtheorem{theorem}{Theorem}[section]
\newaliascnt{lemma}{theorem}

\aliascntresetthe{lemma}
\newaliascnt{corollary}{theorem}
\newtheorem{corollary}[corollary]{Corollary}
\aliascntresetthe{corollary}
\theoremstyle{remark}  
\newaliascnt{remark}{theorem}
\newtheorem{remark}[remark]{Remark}
\aliascntresetthe{remark}
\crefname{theorem}{theorem}{theorems}
\Crefname{theorem}{Theorem}{Theorems}
\crefname{lemma}{lemma}{lemmas}
\Crefname{lemma}{Lemma}{Lemmas}
\crefname{corollary}{corollary}{corollaries}
\Crefname{corollary}{Corollary}{Corollaries}
\crefname{remark}{remark}{remarks}
\Crefname{remark}{Remark}{Remarks}

\begin{document}

\title{Optimal Monte Carlo integration on matrix manifolds}


\ifdefined\arxivversion
    \author[Aleksei G. Sorokin]{Aleksei G. Sorokin}
    \author[Yihan Guo]{Yihan Guo}
    \author[Lek-Heng Lim]{Lek-Heng Lim}
    \thanks{Corresponding author: Aleksei G. Sorokin, \href{mailto:sorokin@uchicago.edu}{\nolinkurl{sorokin@uchicago.edu}}}
    \thanks{All authors: Department of Statistics, University of Chicago, Chicago, IL 60637, USA}
    \thanks{Current address of Yihan Guo: Department of Mathematics, Purdue University, West Lafayette, IN 47907, USA}
\else
    \author{Aleksei G. Sorokin}
    \address{Department of Statistics, University of Chicago, Chicago, IL 60637, USA}
    \email{sorokin@uchicago.edu}
    \thanks{}
    
    \author{Yihan Guo}
    \address{Department of Statistics, University of Chicago, Chicago, IL 60637, USA}
    \curraddr{Department of Mathematics, Purdue University, West Lafayette, IN, 47907, USA}
    \email{guo1009@purdue.edu}
    \thanks{}
    
    \author{Lek-Heng Lim}
    \address{Department of Statistics, University of Chicago, Chicago, IL 60637, USA}
    \email{lekheng@uchicago.edu}
    \thanks{}
\fi


\date{}

\dedicatory{}


\begin{abstract}
    This article develops optimal order (up to logarithmic factors) weighted cubature rules for Monte Carlo numerical integration in Sobolev reproducing kernel Hilbert spaces over common matrix manifolds. Specifically, we consider matrix models of the Stiefel, Grassmann, flag, and Lagrangian Grassmann manifolds each over either real numbers $\bbR$, complex numbers $\bbC$, or quaternions $\bbH$. The proposed cubature rules use evaluation points which are independent and identically distributed uniform on the manifold. The weights are chosen to be optimal with respect to the reproducing kernel Hilbert space worst-case error, under the constraint that the weights sum to one. Crucially, this constraint ensures optimal weights may be found without having to evaluate the kernel mean embedding (the integral of the kernel), which is often difficult to compute. We consider computable Mat\'ern kernels of Frobenius distances between matrices which reproduce the Sobolev reproducing kernel Hilbert spaces with equivalent norms. Open-source numerical experiments validate our theoretical development, including an application to simultaneous approximation of intrinsic volumes in the Cauchy--Kubota formula.  We also present preliminary empirical results beyond independent points including transformations of low-discrepancy quasi-Monte Carlo points from the unit cube and jointly optimized points and weights.
\end{abstract}

\maketitle

\section{Introduction} 

\subsection{Motivation}

Numerical integration methods are foundational to statistical computation. Stiefel, Grassmann, flag, and Lagrangian Grassmann manifolds are important generalizations of widely studied manifolds which have recently been identified with submanifolds of matrix spaces \citep{guo2025eigen,LimYe2024MinimalEquivariant}. This article uses these matrix models to develop optimal order (up to logarithmic factors) cubature routines in Sobolev reproducing kernel Hilbert spaces for numerical integration on these manifolds. Despite their wide-ranging relevance, existing work on numerical integration over these matrix manifolds has largely focused on special cases such as the Grassmann manifold \citep{breger2017quasi} and the unit sphere  \citep{brauchart2012quasi,brauchart2014qmc}, or on general closed manifolds without exploiting matrix structure \citep{ehler2019optimal}. Our generalized approach improves upon existing methods to provide practical, optimal, and theoretically justified cubature schemes on these manifolds.

Numerical integration on Stiefel, Grassmann, flag, and Lagrangian Grassmann manifolds plays a pivotal role in a wide range of modern applications across computational mathematics, statistics, physics, machine learning, and convex geometry. In random matrix theory, integrals over these manifolds naturally arise in the study of eigenvalue distributions and invariant measures \citep{guionnet2002large}. In statistics, such integrals are central to the computation of distribution functions for various test statistics. This is especially relevant to the evaluation of marginal likelihoods and Bayesian model evidence under Wishart priors for covariance matrices, which requires integration over possible eigenstructure manifolds \citep{james1964distributions,muirhead1982aspects}. In physics and quantum information theory, these manifolds appear in state space decompositions, quantum tomography, entanglement analysis, and quantum sensing, where key quantities reduce to integrals over manifolds associated with random symmetric states \citep{bengtsson2006geometry,OszmaniecAugusiakGogolinKolodynskiAcinLewenstein2016}. In signal processing and machine learning, these manifolds are widely used to represent subspaces and filtrations, providing powerful tools for dimensionality reduction and statistical analysis on nonlinear data spaces \citep{edelman1998geometry}. 

In convex geometry, important quantities such as intrinsic volumes may be formulated as integrals over random projections onto nested subspaces characterized via these manifolds. In \Cref{sec:Kubota} we will demonstrate the utility of our proposed cubature schemes by simultaneously approximating multiple intrinsic volumes in the Cauchy--Kubota formula using a single weighted cubature rule  \citep{Hadwiger1957,SchneiderWeil2008,klain.intro_geometric_prob,kubota1925konvex}. We emphasize that such simultaneous approximations will use the same points and weights for each intrinsic volume estimate, which would not be possible using existing cubature schemes on special cases of these manifolds such as the unit spheres \citep{brauchart2012quasi,brauchart2014qmc} or Grassmann manifolds \citep{breger2017quasi}.

\subsection{Matrix manifolds} \label{sec:matrix_manifolds}

Let $\O(n)$ be the orthogonal group in $\bbR^{n \times n}$, let $\U(n)$ be the unitary group in $\bbC^{n \times n}$, and let $\Sp(n)$ be the quaternionic unitary group in $\bbH^{n \times n}$. We denote by $X^\top$, $X^\dagger$, and $X^\h$ the transpose, conjugate transpose, and quaternionic conjugate transpose respectively, and sometimes use $X^*$ to avoid distinguishing between these operations. Note that $\Sp(n)$ may be identified with the space of all $2n \times 2n$ complex matrices $Q$ which are both unitary, $Q^\dagger Q = I_{2n}$, and symplectic, $Q^\top J_{2n} Q = J_{2n}$ where $J_{2n} = \begin{bsmallmatrix} 0 & -I_n \\ I_n & 0 \end{bsmallmatrix}$. Specifically, we identify the quaternionic matrix $X + Y \mathsf{j} \in \bbH^{a \times b}$ with the complex matrix $\begin{bsmallmatrix} X & Y \\ -\overline{Y} & \overline{X} \end{bsmallmatrix} \in \bbC^{2a \times 2b}$ where $\overline{X},\overline{Y} \in \bbC^{a \times b}$ are the conjugates of $X,Y \in \bbC^{a \times b}$ respectively. Let us also use the notations $Q^\s := - J_{2n} Q^\top J_{2n}$ and $Q^\la := I_{n,n} Q^\h I_{n,n}$ where $I_{a,b} := \diag(I_a,-I_b)$. We use the notation $q_1,\dots,q_n$ for the columns of $Q$.

In this article we will consider the Stiefel ($\V$), Grassmann ($\Gr$), flag ($\Flag$), and Lagrangian Grassmann ($\LGr$) manifolds over $\bbR$, $\bbC$, and $\bbH$ as defined in \Cref{tab:manifolds}. The submanifold matrix representations were originally proposed in \citep{guo2025eigen,LimYe2024MinimalEquivariant} and recently used in \cite{LaiLimYe2025EDDManifoldOptimization,LimLuYe2025ProductsGrassmannians}. Beyond the projector model of the Grassmannian used in \citep{breger2017quasi}, our use of these submanifold matrix representations is, to the best of our knowledge, novel within the literature regarding Monte Carlo integration on closed manifolds \citep{brauchart2014qmc,breger2018points,breger2017quasi,briol2015probabilistic,graef2013efficient}. The definitions of the Grassmann manifolds use $\Delta := I_{k,n-k} = \diag(I_k,-I_{n-k})$ where $1 \leq k < n$. The definitions of the flag manifolds use $\Delta_a := \diag(a_1 I_{k_1-k_0},\dots,a_{p+1} I_{k_{p+1}-k_p})$ where $a_1,\dots,a_{p+1} \in \bbR$ are any distinct arbitrary real constants, $0 =: k_0 < k_1 < \cdots < k_p < k_{p+1} := n$, and $n_j = k_j-k_{j-1}$ for $1 \leq j \leq p+1$. The Grassmann manifold is a special case of a flag manifold.  

The intrinsic dimension $d$ of the manifold is the dimension of the top group minus the dimension of the stabilizer group in the homogeneous quotient space using $\dim(\O(n)) = n(n-1)/2$, $\dim(\U(n)) = n^2$, and $\dim(\Sp(n)) = 2n^2+n$. The ambient dimension $r$ is the number of real valued scalars used to represent the matrix where for quaternionic matrices we take this to be the number of real scalars in their complex representation.

For a manifold $\calM$, let $\mu$ denote the uniform (Haar) measure on $\calM$ satisfying $\mu(\calM) = 1$. The definitions of $\mu$ on matrix submanifolds $\calM$ depend on the uniform measures $\unif(\calG)$ where $\calG \in \{\O(n),\U(n),\Sp(n)\}$ is the relevant group. The QR and eigen-decompositions may both be used to generate uniform $Q \sim \unif(\O(n))$ or $Q \sim \unif(\U(n))$. The singular value decomposition may be used to generate uniform $Q \sim \unif(\Sp(n))$. See \citep{guo2025eigen} for details on generating uniform points on these matrix submanifolds. 

\begin{table}[htbp]
    \centering
    \resizebox{\textwidth}{!}{%
    \begin{tabular}{l l l l}
        manifold  & matrix submanifold $\calM$ & \multicolumn{2}{l}{$X \sim \mu$} \\
        \hline 
        $\V(k,\bbR^n)$ & $\left\{X \in \bbR^{n \times k}: X^\top X = I\right\}$ & $[q_1 | \cdots | q_k] \in \bbR^{n \times k}:$ & $Q = [q_1 | \cdots | q_n] \sim \unif(\O(n))$ \\
        $\V(k,\bbC^n)$ & $\left\{X \in \bbC^{n \times k}: X^\dagger X = I\right\}$ & $[q_1 | \cdots | q_k] \in \bbC^{n \times k}:$ & $Q = [q_1 | \cdots | q_n] \sim \unif(\U(n))$ \\
        $\V(k,\bbH^n)$ & $\left\{X \in \bbH^{n \times k}: X^\h X = I\right\}$ & $[q_1 | \cdots | q_k] \in \bbH^{n \times k}:$ & $Q = [q_1 | \cdots | q_n] \sim \unif(\Sp(n))$ \\
        $\Gr(k,\bbR^n)$ & $\left\{Q \Delta Q^\top \in \bbR^{n \times n}: Q \in \O(n)\right\}$ & $Q \Delta Q^\top \in \bbR^{n \times n}:$ & $Q \sim \unif(\O(n))$ \\
        $\Gr(k,\bbC^n)$ & $\left\{Q \Delta Q^\dagger \in \bbC^{n \times n}: Q \in \U(n)\right\}$ & $Q \Delta Q^\dagger \in \bbC^{n \times n}:$ & $Q \sim \unif(\U(n))$ \\
        $\Gr(k,\bbH^n)$ & $\left\{Q \Delta Q^\h \in \bbH^{n \times n}: Q \in \Sp(n)\right\}$ & $Q \Delta Q^\h \in \bbH^{n \times n}:$ & $Q \sim \unif(\Sp(n))$ \\
        $\Flag(k_1,\dots,k_p,\bbR^n)$ & $\left\{Q \Delta_a Q^\top \in \bbR^{n \times n}: Q \in \O(n)\right\}$ & $Q \Delta_a Q^\top \in \bbR^{n \times n}:$ & $Q \sim \unif(\O(n))$ \\
        $\Flag(k_1,\dots,k_p,\bbC^n)$ & $\left\{Q \Delta_a Q^\dagger \in \bbC^{n \times n}: Q \in \U(n)\right\}$ & $Q \Delta_a Q^\dagger \in \bbC^{n \times n}:$ & $Q \sim \unif(\U(n))$ \\
        $\Flag(k_1,\dots,k_p,\bbH^n)$ & $\left\{Q \Delta_a Q^\h \in \bbH^{n \times n}: Q \in \Sp(n)\right\}$ & $Q \Delta_a Q^\h \in \bbH^{n \times n}:$ & $Q \sim \unif(\Sp(n))$ \\
        $\LGr(\bbR^{2n})$ & $\left\{Q Q^\top \in \bbC^{n \times n}: Q \in \U(n)\right\}$ & $Q Q^\top \in \bbC^{n \times n}:$ & $Q \sim \unif(\U(n))$ \\
        $\LGr(\bbC^{2n})$ & $\left\{Q Q^\la \in \bbC^{2n \times 2n}: Q \in \Sp(n)\right\}$ & $Q Q^\la \in \bbC^{2n \times 2n}:$ & $Q \sim \unif(\Sp(n))$ \\
        $\LGr(\bbH^{2n})$ & $\left\{Q Q^\s \in \bbC^{2n \times 2n}: Q \in \U(2n)\right\}$ & $Q Q^\s \in \bbC^{2n \times 2n}:$ & $Q \sim \unif(\U(2n))$ \\
        \hline 
        \hline 
        manifold &  homogeneous & $d$ & $r$\\
        \hline 
        $\V(k,\bbR^n)$ & $\O(n) / \O(n-k)$ & $nk-k(k+1)/2$ & $nk$ \\ 
        $\V(k,\bbC^n)$ & $\U(n) / \U(n-k)$ & $2nk - k^2$ & $2nk$  \\ 
        $\V(k,\bbH^n)$ & $\Sp(n) / \Sp(n-k)$ & $4nk-2k^2+k$ & $8nk$  \\ 
        $\Gr(k,\bbR^n)$ & $\O(n) / (\O(k) \times \O(n-k))$ & $k(n-k)$ & $n^2$  \\ 
        $\Gr(k,\bbC^n)$ & $\U(n) / (\U(k) \times \U(n-k))$ & $2k(n-k)$ & $2n^2$ \\ 
        $\Gr(k,\bbH^n)$ & $\Sp(n) / (\Sp(k) \times \Sp(n-k))$ & $4k(n-k)$ & $8n^2$ \\ 
        $\Flag(k_1,\dots,k_p,\bbR^n)$ & $\O(n) / (\O(k_1) \times \cdots \times \O(n-k_p))$ & $\sum_{1 \leq i < j \leq p+1} n_i n_j$ & $n^2$ \\ 
        $\Flag(k_1,\dots,k_p,\bbC^n)$ & $\U(n) / (\U(k_1) \times \cdots \times \U(n-k_p))$ & $2\sum_{1 \leq i < j \leq p+1} n_i n_j$ & $2 n^2$\\ 
        $\Flag(k_1,\dots,k_p,\bbH^n)$ & $\Sp(n) / (\Sp(k_1) \times \cdots \times \Sp(n-k_p))$ & $4\sum_{1 \leq i < j \leq p+1} n_i n_j$ & $8 n^2$ \\ 
        $\LGr(\bbR^{2n})$ & $\U(n) / \O(n)$ & $n(n+1)/2$ & $2n^2$ \\ 
        $\LGr(\bbC^{2n})$ & $\Sp(n) / \U(n)$ & $n(n+1)$ & $8n^2$ \\ 
        $\LGr(\bbH^{2n})$ & $\U(2n) / \Sp(n)$ & $n(2n-1)$ & $8n^2$
    \end{tabular}
    }
    \caption{Stiefel, Grassmann, flag, and Lagrangian Grassmann manifolds over $\bbR$, $\bbC$, and $\bbH$.}  
    \label{tab:manifolds}
\end{table}

\subsection{Worst-case error of cubature in reproducing kernel Hilbert spaces}

In the more general setting of a smooth, connected, closed Riemannian manifold $\calM$ of (intrinsic) dimension $d$, Monte Carlo methods approximate the integral of some $f: \calM \to \bbR$ with respect to the uniform (Haar) measure $\mu$ on $\calM$ satisfying $\mu(\calM) = 1$ by the weighted cubature rule
\begin{equation*}
    \int_\calM f(X) d\mu(X) \approx \sum_{i=1}^N w_i f(X_i)
\end{equation*}
with points $X_1,\dots,X_N \in \calM$ and weights $w_1,\dots,w_N \in \bbR$. We would like to find points and weights which obtain the optimal-rate worst-case error of the above approximation over functions in the Sobolev reproducing kernel Hilbert space $H(K_\mathrm{B}^{(s)})$ on $\calM$ with smoothness parameter $s>d/2$ and Bessel kernel $K_\mathrm{B}^{(s)}$, see \Cref{sec:Matern_kernels} and \citep{brandolini2010quadrature,ehler2019optimal}. The optimal worst-case error in the Sobolev reproducing kernel Hilbert space $H(K_\mathrm{B}^{(s)})$ is  $\calO(N^{-s/d})$,  and there exists a set of points $X_1,\dots,X_N$ and weights $w_1,\dots,w_N$ which obtains this optimal rate \citep{brandolini2010quadrature}.

In the special case of the Grassmann manifold, \citep{breger2017quasi} considered optimizing points $X_1,\dots,X_N$ with equal weights $w_i = 1/N$ for $1 \leq i \leq N$ with fixed $N$. For the  matrix manifolds in \Cref{sec:matrix_manifolds}, such a procedure would require optimizing over $\calO(Nr)$ parameters which quickly becomes infeasible for moderate $r$ and $N$. Rather than using expensive optimization routines to find optimal high-dimensional points $X_1,\dots,X_N$, one would prefer to use independent and identically distributed uniform points $X_1,\dots,X_N \simiid \mu$ which may be sampled at $\calO(N n^3)$ cost using standard procedures backed by the QR decomposition, eigendecomposition, and singular value decomposition \citep{guo2025eigen}. However, Monte Carlo methods with independent and identically distributed points $X_1,\dots,X_N$ and equal weights $w_i = 1/N$ for $1 \leq i \leq N$ have expected worst-case error of exact order $N^{-1/2}$, see \citep{brauchart2014qmc,breger2018points,hinrichs2010optimal,novak2010tractability,plaskota2009new} and \citep[Corollary 2.8]{graef2013efficient}. Such rules are necessarily worse than the optimal $\calO(N^{-s/d})$ attainable in Sobolev spaces where $s>d/2$. Fortunately, \citep{ehler2019optimal} showed that using independent and identically distributed points $X_1,\dots,X_N$ and worst-case-error-minimizing weights $w_1,\dots,w_N$ in Sobolev spaces achieves the optimal expected worst-case error $\calO(N^{-s/d})$ up to logarithmic factors.

\subsection{Challenges addressed in this article} 

Two central challenges remain in developing a feasible numerical integration scheme on matrix manifolds using independent and identically distributed points and optimized weights: 
\begin{enumerate}
    \item The first challenge is to circumvent the need to evaluate the kernel mean embedding (the integral of the kernel) which arises in the reproducing kernel Hilbert space worst-case error form \eqref{eq:WCE}. Related methods \citep{ehler2019optimal,breger2017quasi,brauchart2012quasi,brauchart2014qmc} have either worked out kernel mean embeddings for special cases of matrix manifolds, resorted to numerical approximations of the embeddings, or optimized bounds on the worst-case error. Here we use the exact worst-case error form and avoid the need to compute the kernel mean embedding by choosing worst-case-error-minimizing weights $w_1,\dots,w_N$ under the restriction $\sum_{i=1}^N w_i = 1$, i.e., the weights sum to $1$ and the cubature scheme exactly integrates constant functions. We show in \Cref{thm:WCE_w_sum1_asymp_WCE_w_any} that, for kernels with constant mean embedding, the worst-case errors of the constrained and unconstrained optimal weights are related by an explicit identity and are asymptotically equivalent, and in \Cref{corr:matern_weights_optimal} that the resulting rule attains the optimal expected rate up to logarithmic factors.
    \item The second challenge is to determine a computable kernel form for Sobolev spaces on the matrix manifolds of interest. Existing work, such as that for numerical integration on the Grassmann manifold \citep{breger2017quasi}, uses polynomial kernels and diffusion kernels which approximate the desired Sobolev space but do not exactly reproduce all functions in the space. Drawing upon work in \citep{de2021reproducing,borovitskiy2020matern,FuselierWright2012}, we prove in \Cref{thm:matern_chordal} that the computable Mat\'ern kernel of the Frobenius distance between matrices exactly reproduces the Sobolev space of interest with an equivalent norm.
\end{enumerate}

\subsection{Outline of remaining sections} 

The remainder of this article is organized as follows. \Cref{sec:methods} details our methodology. Specifically, \Cref{sec:error_weighted_cubature} discusses the error of weighted cubature rules in reproducing kernel Hilbert spaces, and \Cref{sec:Matern_kernels} details the Sobolev reproducing kernel Hilbert spaces of interest and the Mat\'ern kernels with Frobenius distances that reproduce these Sobolev spaces with equivalent norms. \Cref{sec:numerical_experiments} describes our implementation of these methods and gives numerical experiments. Specifically, \Cref{sec:implementation} gives a unified algorithm for generating independent and identically distributed matrix points with optimal weights constrained to sum to one and then links our open-source software package, \Cref{sec:optimal_wce_convergence} validates theoretical convergence rates, and \Cref{sec:Kubota} applies the proposed methods to simultaneous approximation of intrinsic volumes in the Cauchy--Kubota formula. \Cref{sec:conclusion} gives a brief conclusion. 

\section{Methods} \label{sec:methods}

\subsection{Weighted cubature error in reproducing kernel Hilbert spaces} \label{sec:error_weighted_cubature}

The presentation in this section and the next closely follows that of \citep{ehler2019optimal}. Suppose $\calM$ is any smooth, connected, closed Riemannian manifold of (intrinsic) dimension $d$. Let $\mu$ be the normalized Riemannian measure, i.e., the uniform measure or normalized Haar measure on the manifold $\calM$ satisfying $\mu(\calM) = 1$. Suppose $H := H(K)$ is a reproducing kernel Hilbert space of real-valued functions with symmetric positive definite kernel $K: \calM \times \calM \to \bbR$, i.e., $K$ is symmetric, $K(X,Y)=K(Y,X)$ for all $X,Y \in \calM$, and positive definite, $\sum_{i,j=1}^N c_i c_j K(X_i,X_j) > 0$ for all $N \in \bbN$, nonzero $c_1,\dots,c_N \in \bbR$, and distinct points $X_1,\dots,X_N \in \calM$. For points $\calX = (X_1,\dots,X_N) \subset \calM$ and weights $w = (w_1,\dots,w_N) \in \bbR^N$, \citep{hickernell.generalized_discrepancy_quadrature_error_bound} shows that we may write the squared worst-case error over functions in $H$ as 
\begin{equation} \label{eq:WCE} 
\begin{aligned}
    \WCEm{}{\calX}{w}{H}^2 :=& \sup_{\substack{f \in H \\ \llVert f \rrVert_H \leq 1}} \llvert \int_\calM f(X) d\mu(X) - \sum_{i=1}^N w_i f(X_i) \rrvert^2 \\
    =& \sum_{i,j=1}^N w_i w_j K(X_i,X_j)  -2 \sum_{i=1}^N w_i \int_\calM K(X_i,Y) d\mu(Y) \\
    &+ \int_\calM \int_\calM K(X,Y) d\mu(X) d\mu(Y).
\end{aligned}
\end{equation}
If independent and identically distributed uniform points $X_1,\dots,X_N \simiid \mu$ and equal weights $w_i = 1/N$ for $1 \leq i \leq N$ are used, then  \citep{brauchart2014qmc,breger2018points,novak2010tractability,graef2013efficient} have shown that in a reproducing kernel Hilbert space $H$
\begin{equation*}
    \sqrt{\bbE\left[\WCEm{}{\calX}{\vec{1}/N}{H}^2\right]} \asymp N^{-1/2}.
\end{equation*}

Throughout this article we will assume the kernel mean embedding integral $\int_\calM K(X,Y) d\mu(Y)$ in \eqref{eq:WCE} is constant as a function of $X$, which we will denote by 
\begin{equation} \label{eq:kernel_mean_embedding} 
    \gamma := \int_\calM K(X,Y) d\mu(Y) = \int_\calM \int_\calM K(X,Y) d\mu(Y) d\mu(X)
\end{equation} 
for any $X \in \calM$. This property holds whenever a group $\calG$ acts transitively on $\calM$ by $\mu$-preserving maps and $K$ is $\calG$-invariant. For the matrix manifolds $\calM$ in \Cref{sec:matrix_manifolds}, this is the case for any group invariant kernel $K(X,Y) = K(Q \cdot X,Q \cdot Y)$ where $X,Y \in \calM$ and $Q \cdot X$ denotes the group action of an arbitrary group element $Q \in \calG$ on $X$. Specifically, the group action is \begin{equation} \label{eq:group_action} 
    Q \cdot X = \begin{cases} Q X, & \calM = \V \\  Q X Q^*, & \calM \in \{\Gr,\Flag\} \end{cases} 
\end{equation}
with groups $\calG \in \{\O(n),\U(n),\Sp(n)\}$ for respective fields $\{\bbR,\bbC,\bbH\}$, and 
\begin{equation} \label{eq:group_action_2} 
    Q \cdot X = \begin{cases} QXQ^\top, & \calM = \LGr(\bbR^{2n}) \\  QXQ^\la, & \calM = \LGr(\bbC^{2n})\\QXQ^\s, & \calM = \LGr(\bbH^{2n}) \end{cases} 
\end{equation}
with groups $\calG \in \{\U(n),\Sp(n),\U(2n)\}$ for respective fields $\{\bbR,\bbC,\bbH\}$.
 Examples of such group invariant kernels over these matrix manifolds are those which only depend on the Frobenius distance, such as the Mat\'ern kernels of the Frobenius distance we define in \Cref{sec:Matern_kernels}, or those which only depend on the Frobenius inner product, such as the polynomial kernels in \Cref{remark:polynomial_kernels}. An explicit proof of group invariance for kernels of the Frobenius norm on these matrix manifolds is given in the second item of  \Cref{thm:matern_chordal}. A similar proof holds for kernels of the Frobenius inner product on these matrix manifolds. 

Let us denote the $N \times N$ symmetric positive definite Gram matrix by $G \in \bbR^{N \times N}$ where $G_{ij} = K(X_i,X_j)$ for $i,j \in \{1,\dots,N\}$. Then \eqref{eq:WCE}  may be written as 
\begin{equation*}
    \WCEm{}{\calX}{w}{H}^2 = w^\top G w - 2 \gamma w^\top \vec{1} + \gamma,
\end{equation*}
where $\vec{1}$ is the vector of ones. 
For a fixed set of points $\calX$, the optimal unconstrained weights are 
\begin{equation} \label{eq:wstar_any}
    \tw := \argmin_{w \in \bbR^N} \WCEm{}{\calX}{w}{H}^2 = \gamma G^{-1} \vec{1}
\end{equation}
for which 
\begin{equation*}
    \WCEm{}{\calX}{\tw}{H}^2 = \gamma - \gamma^2 \vec{1}^\top G^{-1} \vec{1}.
\end{equation*}
Unlike \citep{ehler2019optimal}, we do not assume the kernel integral $\gamma$ in \eqref{eq:kernel_mean_embedding} is analytically available, thus making $\tw$ in \eqref{eq:wstar_any} infeasible to compute. To avoid evaluating this kernel integral, we constrain the weights to sum to $1$, i.e., $w^\top \vec{1} = \sum_{i=1}^N w_i = 1$, under which
\begin{equation} \label{eq:wstar_sum1}
    \hw := \argmin_{\substack{w \in \bbR^N \\ w^\top \vec{1} = 1}} \WCEm{}{\calX}{w}{H}^2 = \argmin_{\substack{w \in \bbR^N \\ w^\top \vec{1} = 1}} \left[w^\top G w - \gamma\right] = \argmin_{\substack{w \in \bbR^N \\ w^\top \vec{1} = 1}} w^\top G w = \frac{G^{-1} \vec{1}}{\vec{1}^\top G^{-1} \vec{1}}
\end{equation}
and for which 
\begin{equation} \label{eq:WCE_RKHS_m_wstar_sum1}
    \WCEm{}{\calX}{\hw}{H}^2 = \frac{1}{\vec{1}^\top G^{-1} \vec{1}} - \gamma.
\end{equation}
The weights $\hw$ coincide with the Bayes--Sard cubature weights for a constant prior mean \citep{karvonen2018bayes} (equivalently, ordinary kriging), since the kernel mean embedding is constant; the point here is that this constancy removes the need to evaluate it. The following theorem shows that for continuous kernels with a positive constant mean embedding, the worst-case errors of the optimal unconstrained weights, $\tw$, and the optimal weights constrained to sum to 1, $\hw$, are related by an explicit identity. In particular, they agree up to a uniform constant and are asymptotically equivalent.

\begin{theorem} \label{thm:WCE_w_sum1_asymp_WCE_w_any}
    Let $\calM$ be a smooth, connected, closed Riemannian manifold of finite (intrinsic) dimension $d$ and $\mu$ be the uniform measure on $\calM$ with $\mu(\calM) = 1$. Let $H := H(K)$ be a reproducing kernel Hilbert space with real-valued symmetric positive definite kernel $K: \calM \times \calM \to \bbR$. Suppose $K$ is continuous and $\int_\calM K(X,Y)\,d\mu(Y) = \gamma > 0$ for all $X \in \calM$. Then $1 \in H$ with $\llVert 1 \rrVert_H^2 = 1/\gamma$, and for any distinct points $\calX \subset \calM$,
    \begin{equation} \label{eq:WCE_hm_tm_identity}
        \WCEm{}{\calX}{\hw}{H}^2 = \frac{\WCEm{}{\calX}{\tw}{H}^2}{1 - \WCEm{}{\calX}{\tw}{H}^2/\gamma},
    \end{equation}
    so that
    \begin{equation} \label{eq:WCE_hm_tm_bounds}
        \WCEm{}{\calX}{\tw}{H}^2 \le \WCEm{}{\calX}{\hw}{H}^2 \le \frac{\min_i K(X_i,X_i)}{\gamma}\WCEm{}{\calX}{\tw}{H}^2,
    \end{equation}
    and $\WCEm{}{\calX}{\hw}{H}/\WCEm{}{\calX}{\tw}{H} \to 1$ whenever $\WCEm{}{\calX}{\tw}{H} \to 0$.
\end{theorem}
\begin{proof}
    First, we show that $1\in H$ and $\llVert 1\rrVert_H^2 = 1/\gamma$. Since $K$ is continuous on the compact set $\calM\times\calM$ we have $c_K := \sup_{X\in\calM}K(X,X)<\infty$. The map $Y\mapsto K(\cdot,Y)\in H$ is continuous, as
    $$\llVert K(\cdot,Y)-K(\cdot,Y')\rrVert_H^2 = K(Y,Y)-2K(Y,Y')+K(Y',Y'),$$
    and bounded, as
    $$\llVert K(\cdot,Y)\rrVert_H = K(Y,Y)^{1/2}\le c_K^{1/2}.$$
    Hence the Bochner integral $m := \int_\calM K(\cdot,Y)\,d\mu(Y)$ exists in $H$. Because $g\mapsto\langle f,g\rangle_H$ is continuous and linear, the reproducing property gives
    $$\langle f,m\rangle_H = \int_\calM\langle f,K(\cdot,Y)\rangle_H\,d\mu(Y) = \int_\calM f\,d\mu \qquad \text{for all } f\in H.$$
    Taking $f = K(\cdot,X)$ shows $m(X) = \int_\calM K(X,Y)\,d\mu(Y) = \gamma$ for every $X\in\calM$, i.e., $m = \gamma\cdot 1$. Since $\gamma>0$, $1 = m/\gamma\in H$ and
    $$\llVert 1\rrVert_H^2 = \frac{\langle m,m\rangle_H}{\gamma^2} = \frac{1}{\gamma^2}\int_\calM m\,d\mu = \frac{\gamma}{\gamma^2} = \frac{1}{\gamma}.$$

    Next, we show the identity in \eqref{eq:WCE_hm_tm_identity}. As $K$ is positive definite and the points are distinct, $G$ is symmetric positive definite, so $a := \vec{1}^\top G^{-1}\vec{1} > 0$. Substituting \eqref{eq:wstar_any} and \eqref{eq:wstar_sum1} into \eqref{eq:WCE}  gives
    $$\WCEm{}{\calX}{\tw}{H}^2 = \gamma(1-\gamma a), \qquad \WCEm{}{\calX}{\hw}{H}^2 = \frac{1}{a}-\gamma = \frac{1-\gamma a}{a}.$$
    The first equation gives $1-\WCEm{}{\calX}{\tw}{H}^2/\gamma = \gamma a > 0$. Therefore
    $$\frac{\WCEm{}{\calX}{\tw}{H}^2}{1-\WCEm{}{\calX}{\tw}{H}^2/\gamma} = \frac{\gamma(1-\gamma a)}{\gamma a} = \frac{1-\gamma a}{a} = \WCEm{}{\calX}{\hw}{H}^2.$$

    We now show the bounds in \eqref{eq:WCE_hm_tm_bounds}. The lower bound holds because $\tw$ minimizes $\WCEm{}{\calX}{w}{H}$ over all $w\in\bbR^N$, a set that contains $\hw$. For the upper bound, if $\WCEm{}{\calX}{\tw}{H}=0$, then \eqref{eq:WCE_hm_tm_identity} forces $\WCEm{}{\calX}{\hw}{H}=0$ and the bound is trivial. Otherwise, the computation above gives
    $$\frac{\WCEm{}{\calX}{\hw}{H}^2}{\WCEm{}{\calX}{\tw}{H}^2} = \frac{1}{\gamma a}.$$
    Now let $c := G^{-1}\vec{1}$ and $f := \sum_{j=1}^N c_jK(\cdot,X_j)\in H$. Then $f(X_i) = (Gc)_i = 1$ for all $i$ and $\llVert f\rrVert_H^2 = c^\top Gc = a$. By the reproducing property and Cauchy--Schwarz, for each $i$,
    $$1 = f(X_i) = \langle f,K(\cdot,X_i)\rangle_H \le \llVert f\rrVert_H K(X_i,X_i)^{1/2} = a^{1/2}K(X_i,X_i)^{1/2},$$
    so $1/a\le\min_i K(X_i,X_i)$, and the upper bound follows.

    Finally, we show the limit.  By \eqref{eq:WCE_hm_tm_identity}
    $$\frac{\WCEm{}{\calX}{\hw}{H}}{\WCEm{}{\calX}{\tw}{H}} = \bigl(1-\WCEm{}{\calX}{\tw}{H}^2/\gamma\bigr)^{-1/2} \to 1$$
    as $\WCEm{}{\calX}{\tw}{H}\to 0$.
\end{proof}

\subsection{Mat\'ern kernels of Frobenius distance on matrix manifolds} \label{sec:Matern_kernels}

To define appropriate Sobolev reproducing kernel Hilbert spaces, let $\Delta_\calM$ denote the Laplace--Beltrami operator on $\calM$, and let $\varphi_0,\varphi_1,\dots$ be the collection of its orthonormal eigenfunctions with respective eigenvalues $-\lambda_0,-\lambda_1,\dots$ arranged by $0=\lambda_0 \leq \lambda_1 \leq \cdots$. We choose $\varphi_\ell$ with $\ell \geq 0$ to be real-valued with $\varphi_0 \equiv 1$. The Fourier transform of $f \in L_2$ is defined by 
$$\hf(\ell) := \int_\calM f(X) \varphi_\ell(X) d\mu(X), \qquad \ell=0,1,2,\dots$$
with the usual extension to distributions on $\calM$. For $s>d/2$, let $H^s(\calM) := H(K_\mathrm{B}^{(s)})$ denote the Sobolev reproducing kernel Hilbert space with Bessel kernel
\begin{equation} \label{eq:Bessel_kernel}
    K_\mathrm{B}^{(s)}(X,Y) = \sum_{\ell=0}^\infty (1+\lambda_\ell)^{-s} \varphi_\ell(X) \varphi_\ell(Y)
\end{equation}
so that $H(K_\mathrm{B}^{(s)})$ consists of all distributions on $\calM$ with $(I-\Delta_\calM)^{s/2} f \in L_2$, i.e., with 
\begin{equation} \label{eq:sobolev_norm}
    \llVert f \rrVert_{H(K_\mathrm{B}^{(s)})} := \llVert (I-\Delta_\calM)^{s/2} f \rrVert_{L_2} = \llVert \sum_{\ell=0}^\infty (1+\lambda_\ell)^{s/2} \hf(\ell) \varphi_\ell \rrVert_{L_2} < \infty.
\end{equation} 
Each function in $H(K_\mathrm{B}^{(s)})$ is continuous, see \citep{brandolini2010quadrature} and \citep[Theorem 7.4.5, Section 7.4.2]{triebel.theory_function_spaces_2}, so that point evaluation is well-defined. For any $s>d/2$ and any set of points $X_1,\dots,X_N \in \calM$ with arbitrary weights $w_1,\dots,w_N \in \bbR$, we have 
\begin{equation} \label{eq:WCE_lower_bound_Sobolev} 
    N^{-s/d} \lesssim \WCEm{}{\calX}{w}{H(K_\mathrm{B}^{(s)})},
\end{equation}
see \citep{brauchart2014qmc} for the sphere and \citep{brandolini2010quadrature} for the general case. Note that the constant in the above equation may depend on $s$ and $\calM$. 

Let us define the covering radius 
\begin{equation*}
    \rho_N := \max_{X \in \calM} \min_{i=1,\dots,N} \mathrm{dist}_\calM(X,X_i)
\end{equation*}
where $\mathrm{dist}_\calM$ denotes the geodesic distance. According to \citep{breger2018points}, it also holds that 
$$\rho_N^{s+d/2} \lesssim \WCEm{}{\calX}{w}{H(K_\mathrm{B}^{(s)})}.$$
Note that the assumptions of \Cref{thm:WCE_w_sum1_asymp_WCE_w_any} always hold for the Bessel kernel when $s>d/2$: $K_\mathrm{B}^{(s)}$ is continuous, and since $\int_\calM\varphi_\ell\,d\mu = \langle\varphi_\ell,\varphi_0\rangle_{L_2} = \delta_{\ell0}$, \eqref{eq:Bessel_kernel} gives $\int_\calM K_\mathrm{B}^{(s)}(X,Y)\,d\mu(Y) = 1$, i.e., $\gamma = 1$. The following result is an immediate consequence of combining \Cref{thm:WCE_w_sum1_asymp_WCE_w_any} with \citep[Theorem 1]{ehler2019optimal}, which is stated for nonnegative optimal weights but extends to the unconstrained optimal weights $\tw$ since these have no larger worst-case error \citep[Remarks 2 and 4]{ehler2019optimal}.

\begin{theorem} \label{thm:Ehler_Thm1_RKHS}
    Let $s>d/2$ and $X_1,\dots,X_N \in \calM$ be any set of distinct points with covering radius $\rho_N$. Then the optimized weights constrained to sum to $1$ in \eqref{eq:wstar_sum1}, computed with $K_\mathrm{B}^{(s)}$, satisfy
    \begin{equation*}
        \WCEm{}{\calX}{\hw}{H(K_\mathrm{B}^{(s)})} \lesssim \rho_N^s.
    \end{equation*}
\end{theorem}

For random points $X_1,\dots,X_N \in \calM$ and any type of weights $w_1,\dots,w_N \in  \bbR$, no matter if random or not, \eqref{eq:WCE_lower_bound_Sobolev} implies that for any $q>0$
\begin{equation*}
    N^{-s/d} \lesssim \left(\bbE\left[\WCEm{}{\calX}{w}{H(K_\mathrm{B}^{(s)})}^q\right]\right)^{1/q}
\end{equation*}
where the constant may depend on $s$ and $\calM$. The following result is an immediate consequence of combining \Cref{thm:WCE_w_sum1_asymp_WCE_w_any} with \citep[Corollary 1 and Remark 4]{ehler2019optimal}.

\begin{corollary} \label{corr:Ehler_Corr1_RKHS}
    Let $X_1,\dots,X_N \in \calM$ be random points, independently distributed according to $\mu$, and suppose $s>d/2$. Then using the optimized weights constrained to sum to $1$ in \eqref{eq:wstar_sum1}, computed with $K_\mathrm{B}^{(s)}$, for each $q \geq 1/s$, it holds that
    \begin{equation*}
        \left(\bbE\left[\WCEm{}{\calX}{\hw}{H(K_\mathrm{B}^{(s)})}^q\right]\right)^{1/q} \lesssim N^{-s/d} \log(N)^{s/d}.
    \end{equation*}
\end{corollary}
Concretely, \Cref{corr:Ehler_Corr1_RKHS} shows that independent points and optimal weights constrained to sum to $1$ achieve the optimal rate $N^{-s/d}$ up to the logarithmic factor $\log(N)^{s/d}$, where the constant in the above equation may depend on $s$, $\calM$, and $q$.

Deriving the eigensystem of the Laplace--Beltrami operator and then determining an analytic expression for the Bessel kernel in \eqref{eq:Bessel_kernel} is often infeasible. To overcome this, we note that if $K_\mathrm{M}^{(s)}$ is a reproducing kernel Hilbert space kernel for which $H(K_\mathrm{B}^{(s)}) = H(K_\mathrm{M}^{(s)})$ (i.e., it reproduces the space) then the norms are necessarily equivalent,  $c_1 \llVert f \rrVert_{H(K_\mathrm{M}^{(s)})} \leq \llVert f \rrVert_{H(K_\mathrm{B}^{(s)})} \leq c_2 \llVert f \rrVert_{H(K_\mathrm{M}^{(s)})}$ for some $c_1,c_2 >0$, and 
\begin{equation*}
    \WCEm{}{\calX}{w}{H(K_\mathrm{B}^{(s)})} \asymp \WCEm{}{\calX}{w}{H(K_\mathrm{M}^{(s)})}.
\end{equation*}
We show in \Cref{thm:matern_chordal} that for the matrix manifolds $\calM$ in \Cref{sec:matrix_manifolds}, one such reproducing and norm-equivalent kernel is the Mat\'ern kernel of the Frobenius distance with $\nu = s-d/2>0$, lengthscale $\kappa>0$, and general scale $\sigma^2>0$:
\begin{equation} \label{eq:Matern_K}
    K_\mathrm{M}^{(s)}(X,Y) = \sigma^2 \tK^{(s)}_\mathrm{M}( \llVert X-Y \rrVert_\mathrm{F}), \qquad \tK_\mathrm{M}^{(s)}(Z) = \frac{2^{1-\nu}}{\Gamma(\nu)} \left(\sqrt{2 \nu} \frac{Z}{\kappa} \right)^\nu K_\nu\left(\sqrt{2 \nu} \frac{Z}{\kappa} \right).
\end{equation}
Here $\Gamma$ is the Gamma function, $K_\nu$ is the  modified Bessel function of the second kind, and $\llVert \cdot \rrVert_\mathrm{F}$ is the Frobenius distance. We remark that when $\nu=s-d/2$ is a half-integer we have 
\begin{equation*}
    \tK_\mathrm{M}^{(s)}(Z) = \exp\left(-\sqrt{2 \nu}\kappa^{-1} Z\right) \sum_{j=0}^{\nu-1/2} c_{\nu,j} (\kappa^{-1} Z)^j 
\end{equation*}
where $c_\nu \in \bbR_{>0}^{\nu+1/2}$ are known scalar coefficients, and the Gaussian kernel is the limiting case 
$$\lim_{\nu \to \infty} \tK_\mathrm{M}^{(s)}(Z) = \exp(- (Z/\kappa)^2 / 2).$$

The following theorem shows the Mat\'ern kernel of Frobenius distance in \eqref{eq:Matern_K} reproduces the Sobolev space with the Bessel kernel in \eqref{eq:Bessel_kernel} with equivalent norms, as well as showing that the conditions of \Cref{thm:WCE_w_sum1_asymp_WCE_w_any} are satisfied. In particular, the constant in \Cref{thm:WCE_w_sum1_asymp_WCE_w_any} is explicit and independent of $\sigma^2$.

\begin{theorem} \label{thm:matern_chordal}
     For all the matrix manifolds in \Cref{tab:manifolds}, the Mat\'ern kernel $K_\mathrm{M}^{(s)}$ in \eqref{eq:Matern_K} satisfies the following properties when $\nu=s-d/2>0$.  
     \begin{enumerate}
        \item $H(K_\mathrm{M}^{(s)})$ reproduces the Sobolev reproducing kernel Hilbert space $H^s(\calM)$ with equivalent norms. 
        \item $\int_\calM K_\mathrm{M}^{(s)}(X,Y) d\mu(Y)$ is constant as a function of $X$. 
        \item $K_\mathrm{M}^{(s)}$ is continuous and strictly positive definite on $\calM$, and $K_\mathrm{M}^{(s)}(X,X) = \sigma^2$ for all $X\in\calM$.
        \item The constant in item (2) is $\gamma = \sigma^2\tgamma$, where $$\tgamma := \int_\calM \tK_\mathrm{M}^{(s)}(\llVert X-Y\rrVert_\mathrm{F})\,d\mu(Y)\in(0,1]$$ for any $X\in\calM$. Consequently, \Cref{thm:WCE_w_sum1_asymp_WCE_w_any} applies with $\llVert 1\rrVert_{H(K_\mathrm{M}^{(s)})}^2 = 1/(\sigma^2\tgamma)$ and
        $$\WCEm{}{\calX}{\hw}{H(K_\mathrm{M}^{(s)})}^2 \le \tgamma^{-1}\,\WCEm{}{\calX}{\tw}{H(K_\mathrm{M}^{(s)})}^2.$$
    \end{enumerate}
\end{theorem}
\begin{proof}
    Proofs of each item are below.
    \begin{enumerate}
        \item Let $K_{\mathrm{B}'}^{(\nu+r/2)}$ denote the Bessel potential kernel on $\bbR^r$, whose reproducing kernel Hilbert space is $H^{\nu+r/2}(\bbR^r)$, and write $H(K)|_\calM$ for the reproducing kernel Hilbert space of $K$ restricted to $\calM\times\calM$. Then
        $$H(K_\mathrm{M}^{(s)}) = H(K_{\mathrm{B}'}^{(\nu+r/2)})|_\calM = H(K_\mathrm{B}^{((\nu+r/2)-(r-d)/2)}) = H(K_\mathrm{B}^{(s)}),$$
        where the first equality follows from the Fourier transform of $K_\mathrm{M}^{(s)}$ in \eqref{eq:Matern_K} \citep[Chapter 4]{rasmussen.gp4ml}, and the second equality follows from \citep[Proposition~2]{FuselierWright2012}.
        \item  The second item follows immediately from the group invariance of Frobenius distance for these matrix manifolds. Specifically, for any $X,X' \in \calM$ there exists a $Q \in \calG$ such that $X'=Q \cdot X$ where $Q \cdot X$ is the group action in \eqref{eq:group_action} or \eqref{eq:group_action_2}. Letting $Y' = Q \cdot Y$, notice that $\mu(dY) = \mu(dY')$ and $\llVert Q \cdot X - Q \cdot Y \rrVert_\mathrm{F} = \llVert X-Y \rrVert_\mathrm{F}$ for any $Q \in \calG$ and $X,Y \in \calM$. Then the result follows from the invariance 
        $$\int_\calM f(\llVert X'-Y' \rrVert_\mathrm{F}) d\mu(Y') = \int_\calM f(\llVert X - Y \rrVert_\mathrm{F}) d\mu(Y)$$
        for any $X,X' \in \calM$, which implies the integral in item (2) is constant as a function of $X$.
        \item The Mat\'ern function $\tK_\mathrm{M}^{(s)}$ is continuous on $[0,\infty)$ with $\tK_\mathrm{M}^{(s)}(0) = 1$, so $K_\mathrm{M}^{(s)}$ is continuous and $K_\mathrm{M}^{(s)}(X,X) = \sigma^2$. The Mat\'ern kernel is strictly positive definite on $\bbR^r$, since its spectral density is positive everywhere \citep[Chapter 4]{rasmussen.gp4ml}. Hence it is strictly positive definite on any subset of $\bbR^r$, in particular on $\calM$.
        \item By item (2), $\gamma = \int_\calM K_\mathrm{M}^{(s)}(X,Y)\,d\mu(Y)$ for any fixed $X\in\calM$. Since $K_\nu(z)>0$ for $z>0$, $\tK_\mathrm{M}^{(s)}>0$; since $K_\mathrm{M}^{(s)}$ is positive definite, $\lvert K_\mathrm{M}^{(s)}(X,Y)\rvert \le (K_\mathrm{M}^{(s)}(X,X)K_\mathrm{M}^{(s)}(Y,Y))^{1/2} = \sigma^2$, so $\tK_\mathrm{M}^{(s)}\le1$. Hence $\tgamma\in(0,1]$ and $\gamma=\sigma^2\tgamma>0$. Together with items (2)--(3), this verifies the hypotheses of \Cref{thm:WCE_w_sum1_asymp_WCE_w_any}, whose upper bound with $\min_i K_\mathrm{M}^{(s)}(X_i,X_i)=\sigma^2$ gives the stated inequality.
    \end{enumerate}
\end{proof}

\begin{corollary} \label{corr:matern_weights_optimal}
    Let $\calM$ be any matrix manifold in \Cref{tab:manifolds}, let $\nu = s-d/2>0$, and let $\hw$ be the weights in \eqref{eq:wstar_sum1} computed from the Mat\'ern Gram matrix $G_{ij} = K_\mathrm{M}^{(s)}(X_i,X_j)$. If $X_1,\dots,X_N \simiid \mu$, then for each $q \geq 1/s$,
    \begin{equation*}
        \left(\bbE\left[\WCEm{}{\calX}{\hw}{H(K_\mathrm{B}^{(s)})}^q\right]\right)^{1/q} \lesssim N^{-s/d}\log(N)^{s/d},
    \end{equation*}
    and the same bound holds with $H(K_\mathrm{B}^{(s)})$ replaced by $H(K_\mathrm{M}^{(s)})$.
\end{corollary}
\begin{proof}
    Let $\mathrm{WCE}_\mathrm{B}(w)$ and $\mathrm{WCE}_\mathrm{M}(w)$ denote the worst-case errors of $(\calX,w)$ in $H(K_\mathrm{B}^{(s)})$ and $H(K_\mathrm{M}^{(s)})$, respectively, and let $\hw_\mathrm{B}$ denote the weights in \eqref{eq:wstar_sum1} computed with $K_\mathrm{B}^{(s)}$. The norm equivalence $c_1\llVert f\rrVert_{H(K_\mathrm{M}^{(s)})} \le \llVert f\rrVert_{H(K_\mathrm{B}^{(s)})} \le c_2\llVert f\rrVert_{H(K_\mathrm{M}^{(s)})}$ from \Cref{thm:matern_chordal} gives $\mathrm{WCE}_\mathrm{B}(w) \le c_1^{-1}\mathrm{WCE}_\mathrm{M}(w)$ and $\mathrm{WCE}_\mathrm{M}(w) \le c_2\mathrm{WCE}_\mathrm{B}(w)$ for every $w\in\bbR^N$. Since $\hw$ minimizes $\mathrm{WCE}_\mathrm{M}$ over weights summing to one and $\hw_\mathrm{B}$ sums to one,
    \begin{equation*}
        \mathrm{WCE}_\mathrm{B}(\hw) \le c_1^{-1}\mathrm{WCE}_\mathrm{M}(\hw) \le c_1^{-1}\mathrm{WCE}_\mathrm{M}(\hw_\mathrm{B}) \le \frac{c_2}{c_1}\mathrm{WCE}_\mathrm{B}(\hw_\mathrm{B}).
    \end{equation*}
    Similarly, $\mathrm{WCE}_\mathrm{M}(\hw) \le \mathrm{WCE}_\mathrm{M}(\hw_\mathrm{B}) \le c_2\mathrm{WCE}_\mathrm{B}(\hw_\mathrm{B})$. Taking $q$-th moments and applying \Cref{corr:Ehler_Corr1_RKHS} completes the proof.
\end{proof}

\begin{remark}[Application to polynomial  kernels] \label{remark:polynomial_kernels}
    The reproducing kernel Hilbert space polynomial kernel $K_\mathrm{P}^{(t)}: \calM \times \calM \to \bbR$ of degree $t \in \bbN$ on a matrix manifold $\calM$ from \Cref{sec:matrix_manifolds} is 
    \begin{equation} \label{eq:poly_kernel}
        K_\mathrm{P}^{(t)}(X,Y) = (C+ \llangle X, Y \rrangle_\mathrm{F})^t, \quad \llangle X,Y \rrangle_\mathrm{F} = \mathrm{Real}\left(\trace(X^* Y)\right)
    \end{equation}
    where $\llangle \cdot, \cdot \rrangle_\mathrm{F}$ is the Frobenius inner product and we require the constant $C > 0$ so $H(K_\mathrm{P}^{(t)})$ includes constant functions. As $H(K^{(t)}_\mathrm{P})$ is a reproducing kernel Hilbert space of finite dimension $D\le\binom{r+t}{t}$, $K^{(t)}_\mathrm{P}$ is not strictly positive definite and $G$ is singular for $N>D$; there $G^{-1}$ in \eqref{eq:wstar_sum1} must be replaced by a pseudo-inverse (or a unisolvent subset of $D$ points used). With worst-case-error-minimizing weights, the worst-case error vanishes once the points are unisolvent for $H(K_\mathrm{P}^{(t)})$, which for independent and identically distributed points happens almost surely when $N \geq D$; for $N<D$ the expected worst-case error is $\calO(N^{-1/2})$.
 In contrast, independent and identically distributed points with equal weights $w_i=1/N$ for $1 \leq i \leq N$ have expected worst-case error $\calO(N^{-1/2})$ for any $N$, even when $N \geq D$.
\end{remark}

\section{Numerical Experiments} \label{sec:numerical_experiments}

\subsection{Implementation} \label{sec:implementation}

An open-source implementation of the following experiments is available in the ``random matrix quasi-Monte Carlo Python package'' \texttt{rmatqmcpy}, which is readily installed with \texttt{pip install rmatqmcpy} and hosted at \url{https://github.com/alegresor/rmatqmcpy}. This section first discusses numerical validation of optimal convergence rates in \Cref{sec:optimal_wce_convergence} and then discusses an application to the Cauchy--Kubota formula for simultaneous calculation of intrinsic volumes in \Cref{sec:Kubota}.

\Cref{algo:opt_weights} gives the complete procedure for computing independent points and optimal weights. In total, this procedure requires $\calO(Nn^3+N^2r+N^3)$ computations and $\calO(Nr+N^2)$ storage, assuming $K_\mathrm{M}^{(s)}$ costs $\calO(r)$ per evaluation. The same procedure and costs hold for the polynomial kernel $K^{(t)}_\mathrm{P}$ in \eqref{eq:poly_kernel} when $N\le D$; for $N>D$ see \Cref{remark:polynomial_kernels}.

The $\calO(n^3)$ cost of drawing each sample matrix may sometimes be reduced. For example, Stiefel and Grassmann manifolds constructed using the QR decomposition only require $\calO(n k^2)$ computations per draw, which is supported in our implementation. Specifically, for the Grassmann case one may use the fact that for $Q = \begin{bsmallmatrix} Q_1 & Q_2 \end{bsmallmatrix}$ with $Q_1 \in  \bbF^{n \times k}$ and $Q_2 \in \bbF^{n \times (n-k)}$ where $\bbF \in \{\bbR,\bbC,\bbH\}$, we have $I = Q Q^* = Q_1 Q_1^* + Q_2 Q_2^*$ so $Q \Delta Q^* = Q_1 Q_1^* - Q_2 Q_2^* = 2 Q_1 Q_1^* - I$ which enables only generating $Q_1$ at cost $\calO(nk^2)$; moreover $\lVert X_i-X_j\rVert_\mathrm{F}^2=8(k-\lVert Q_{1,i}^*Q_{1,j}\rVert_\mathrm{F}^2)$ (with $k$ replaced by $2k$ in the complex representation when $\bbF=\bbH$), so the Gram matrix costs $\calO(N^2nk^2)$ without forming $X_i$, which is beneficial when $k^2<n$. Similar savings are available for flag manifolds by exploiting block matrix structures and the QR decomposition, although we do not consider this more intricate construction here or in the implementation. Note that these savings are only available by exploiting the nested subspace hierarchy produced by the QR decomposition. Such savings are not available for eigendecompositions, which may be used to uniformly generate samples on $\O(n)$ and $\U(n)$, nor for SVD decompositions which may be used to uniformly generate samples on $\Sp(n)$, as in our implementation.

\begin{algorithm}[htbp]
    \caption{Computation of independent and identically distributed points $X_1,\dots,X_N$ and optimal weights $\hw_1,\dots,\hw_N$}
    \label{algo:opt_weights}
    \begin{algorithmic}
        \Require $\nu = s-d/2 >0$ \Comment{Mat\'ern kernel smoothness parameter}
        \Require $N \in \bbN$ \Comment{Number of points}
        \State Sample independent and identically distributed $X_1,\dots,X_N$ uniformly on $\calM$ \\
        \Comment{See \citep{guo2025eigen}: $\calO(Nn^3)$ computations, $\calO(N r)$ storage}
        \State $G_{ij} \gets K_\mathrm{M}^{(s)}(X_i,X_j)$ for $i,j \in \{1,\dots,N\}$ \\
        \Comment{Gram matrix $G \in \bbR^{N \times N}$: $\calO(N^2r)$ computations, $\calO(N^2)$ storage}
        \State $w \gets G^{-1} \vec{1}$ via Cholesky $G=LL^\top$ \\
        \Comment{Unnormalized weights ($\propto\tw$): $\calO(N^3)$ computations, $\calO(N)$ storage}
        \State $\hw \gets w / \vec{1}^\top w$ \Comment{Optimal constrained weights: $\calO(N)$ computations, $\calO(N)$ storage}
    \end{algorithmic}
\end{algorithm}

Reporting worst-case errors requires the kernel integral $\gamma$ in \eqref{eq:kernel_mean_embedding}. By \Cref{thm:matern_chordal},
$$\gamma = \int_\calM K_\mathrm{M}^{(s)}(X,Y)\,d\mu(Y)$$
for any fixed $X\in\calM$, which we estimate by Monte Carlo, choosing the number of samples separately for each experiment according to the cost of kernel evaluations and the accuracy required. Since $\WCEm{}{\calX}{\hw}{H}^2 = 1/(\vec1^\top G^{-1}\vec1)-\gamma$ involves cancellation, any error in the estimate of $\gamma$ shifts the reported squared worst-case error by the same amount. We emphasize that $\gamma$ is used only to report worst-case errors and not to compute the points or weights. All reported errors are means over $25$ independent trials for the worst-case error experiments in \Cref{fig:WCE} and over $100$ independent trials for the intrinsic volume experiments in \Cref{fig:Vk.cube}. In all experiments we set $\sigma^2=1$ and choose $\kappa$ to be the median of $\llVert X - Y \rrVert_\mathrm{F}$ for independent $X,Y \sim \mu$. We note that the sample size at which asymptotic rates appear to take hold depends strongly on $\kappa$.

\subsection{Optimal worst-case error convergence} \label{sec:optimal_wce_convergence}

\begin{figure}[htbp]
    \centering
    \includegraphics[width=.85\linewidth]{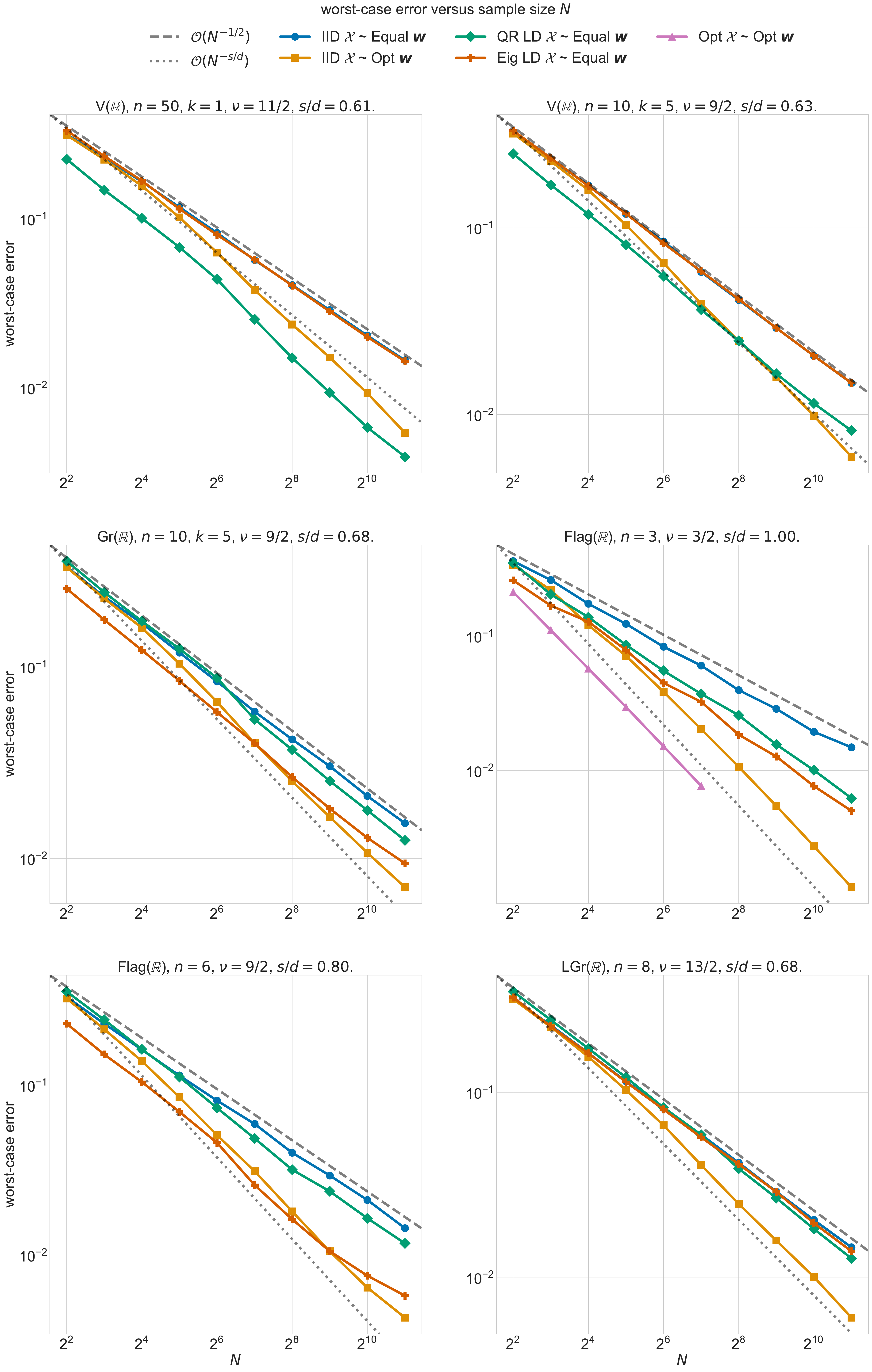}
    \caption{Worst-case error versus $N$ for cubature rules on select real matrix manifolds.}
    \label{fig:WCE}
\end{figure}

In \Cref{fig:WCE} we compare the worst-case error of various numerical integration schemes for select matrix manifolds from \Cref{tab:manifolds}. Here we only consider real manifolds, but our implementation supports the complex and quaternionic cases as well. The real flag manifolds use $\Delta_a = \diag(a)$ with $a = (1,2,3)^\top/\lVert(1,2,3)\rVert_2$ and $a = (1,\dots,6)^\top/\lVert(1,\dots,6)\rVert_2$ respectively. Other parameters are given in the plot titles. In all cases we see that independent and identically distributed $X_1,\dots,X_N$ and equal weights $w_i = 1/N$ for $1 \leq i \leq N$ obtain the theoretical expected worst-case error convergence rate of $\calO(N^{-1/2})$, while independent and identically distributed $X_1,\dots,X_N$ and optimal weights $w_1,\dots,w_N$ restricted to sum to $1$ achieve the optimal $\calO(N^{-s/d})$ expected worst-case error rate. 

In these worst-case error convergence tests, we additionally consider using transformed low-discrepancy points $X_1,\dots,X_N$ in place of independent and identically distributed points while still using equal weights $w_i = 1/N$ for $1 \leq i \leq N$. Specifically, we transform randomized low-discrepancy base-two digital sequences with left linear matrix scrambles and digital shifts \citep{dick2010digital}. Such sequences have been extensively used in quasi-Monte Carlo  methods where the more uniform coverage of low-discrepancy sequences compared to independent and identically distributed sequences leads to improved convergence rates. With equal weights, low-discrepancy points improve upon independent points in the flag, Grassmann, and Stiefel cases, in some cases only in the rate constant and in others also in the empirical rate over the tested range (e.g., QR-based points on $\V(1,\bbR^{50})$ and both constructions on the complete flag with $n=3$). We make no claim about the asymptotic rate. For the Lagrangian Grassmann manifolds, low-discrepancy points and equal weights performed essentially the same as independent and identically distributed points and equal weights. 

Additionally, we notice that there is a delicate tradeoff governing the effectiveness of the QR decomposition versus the eigendecomposition for low-discrepancy points with equal weights. For quasi-Monte Carlo, the efficiency of low-discrepancy points depends on the effective dimension of the integrand and the order of coordinate inputs relative to their importance \citep{dick2013high}. At a high level, quasi-Monte Carlo methods are more effective when the variation of the integrand rapidly decays with the input index. Eigendecompositions generally pair nicely with quasi-Monte Carlo methods as they place greater importance on earlier indices. However, here we are using eigendecompositions to generate $\O(n)$ and $\U(n)$ matrices which are then used to build the submanifolds matrices which are then passed through the Mat\'ern kernels. For complete flag manifolds, we found the eigen-based method provides significant savings over the QR-based method.

In contrast, for the Grassmann and Stiefel manifolds the QR decompositions may reduce the total number of low-discrepancy sequence dimensions from $n^2$ to $nk$. For small $k$, as in the Stiefel manifold example with $k=1$, we found the QR-based method significantly outperformed the eigen-based method, while for $\Gr(5,\bbR^{10})$ the eigen-based method performed better. We leave detailed analysis of the tradeoff between QR-based and eigen-based decompositions with low-discrepancy points to future work.

Finally, for the flag manifold case with $n=3$, we considered jointly optimizing both the points and the weights to minimize the worst-case error subject to weights summing to $1$. Specifically, we optimize the points $X_1,\dots,X_N$ to minimize \eqref{eq:WCE_RKHS_m_wstar_sum1} where the optimal weights may then be computed by \eqref{eq:wstar_sum1}. Notice that if the Gram matrix has a Cholesky decomposition $G=L L^\top$, then minimizing \eqref{eq:WCE_RKHS_m_wstar_sum1} is equivalent to maximizing $\vec{1}^\top G^{-1} \vec{1} = \lVert L^{-1} \vec{1} \rVert_2^2$. We parameterize each point by an unconstrained matrix $Z_i\in\bbR^{n\times n}$ whose QR factor $Q_i$ (with positive diagonal in $R_i$) defines $X_i = Q_i\Delta_aQ_i^\top$, initializing $Z_1,\dots,Z_N$ with the QR-based transformation of low-discrepancy points described above. We then maximize $\lVert v\rVert_2^2$ for the residual vector $v = L^{-1}\vec1 \in \bbR^N$ using an adaptive Levenberg--Marquardt scheme \citep{levenberg1944method,marquardt1963algorithm,more1978levenberg}, where for numerical stability $L$ is the Cholesky factor of $G + \epsilon I$ with $\epsilon = 10^{-8}$. The optimized points and weights obtain the optimal rate $\calO(N^{-s/d})$ and significantly improve upon the rate constant for independent and identically distributed points with optimal weights. However, optimizing over $\calO(N r)$ parameters quickly becomes infeasible in larger scale settings, especially when using a Levenberg--Marquardt scheme which requires $\calO(N^3 r^3)$ operations per step. 

\subsection{Application to the Cauchy--Kubota formula for intrinsic volumes} \label{sec:Kubota}

\begin{figure}[htbp]
    \centering
    \includegraphics[width=1\linewidth]{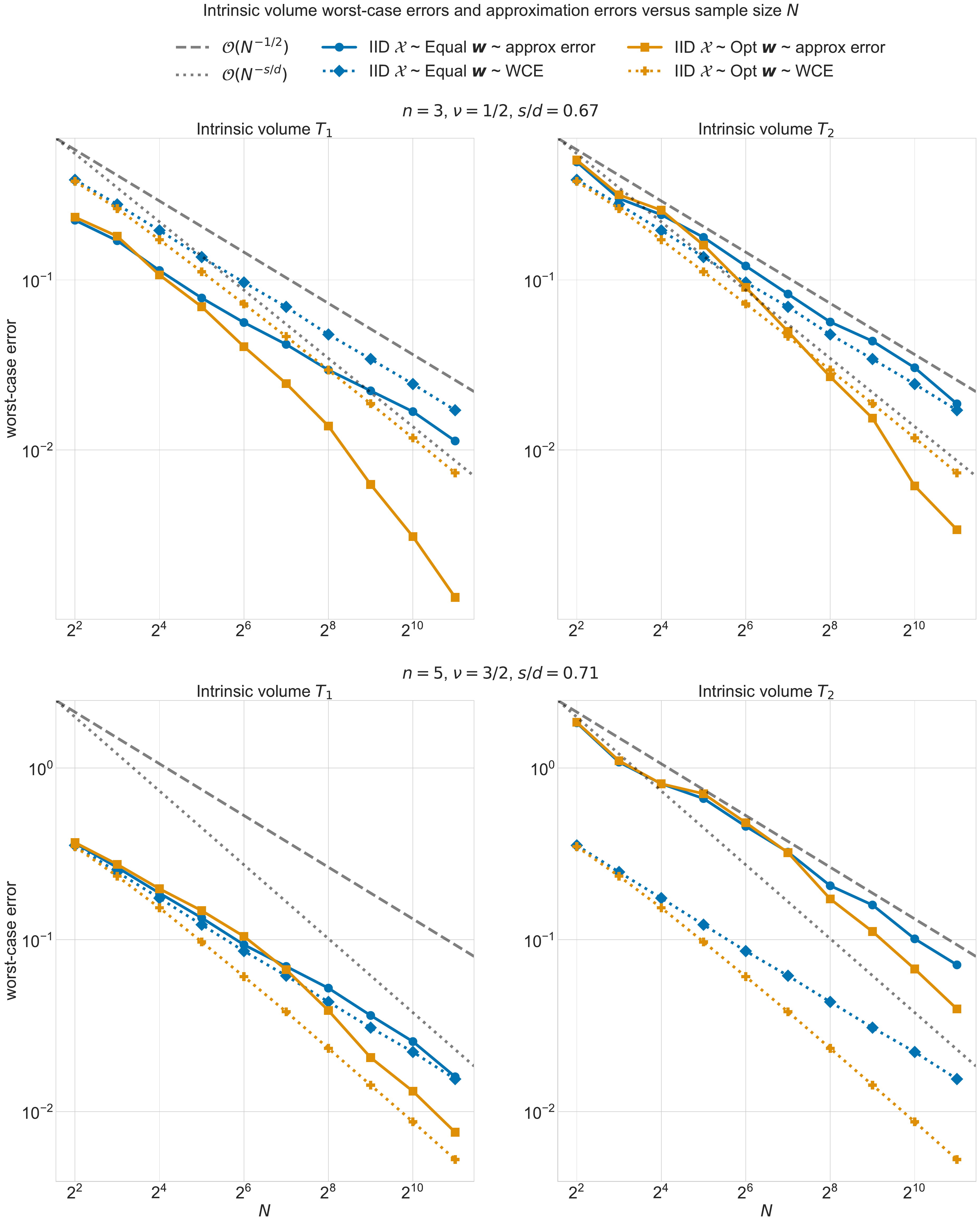}
    \caption{Intrinsic volume computations for the hypercube $[-1,1]^n$ with $n=3$ (top row) and $n=5$ (bottom row). Worst-case errors (dotted) depend only on the points and weights and are therefore identical across columns.}
    \label{fig:Vk.cube}
\end{figure}

Let $\calZ$ be a convex body and let $\calB^n$ denote the unit $n$-ball in $\bbR^n$. The intrinsic volumes $T_\ell(\calZ)$ are defined by 
$$\mathrm{Vol}_n(\calZ+t\calB^n) = \sum_{\ell=0}^n T_\ell(\calZ) B_{n-\ell}t^{n-\ell} \qqtqq{where} B_j :=  \mathrm{Vol}_j(\calB^j) = \frac{\pi^{j/2}}{\Gamma\left(\frac{j}{2}+1\right)}$$
and $\Gamma$ is the Gamma function. If $\tcalZ$ is an $\ell$-dimensional convex body, then $T_\ell(\tcalZ) = \mathrm{Vol}_\ell(\tcalZ)$, i.e., the intrinsic volume equals the actual volume. The Cauchy--Kubota formula \citep{Hadwiger1957,SchneiderWeil2008,klain.intro_geometric_prob,kubota1925konvex} states that if $\calZ \subset \bbR^n$ is a convex body, then for $\ell \leq n$ we have 
$$T_\ell(\calZ) = \binom{n}{\ell} \frac{B_n}{B_{n-\ell} B_\ell} \int_{\calM_\ell} \mathrm{Vol}_\ell(\mathrm{Proj}_{[q_1 | \cdots | q_\ell]}(\calZ)) \mu_\ell(dX)$$
where $\mu_\ell$ is the normalized uniform measure on $\calM_\ell = \Gr(\ell,\bbR^n)$ and, for $X = Q\Delta Q^\top\in\calM_\ell$, $\mathrm{Proj}_{[q_1 | \cdots | q_\ell]}(\calZ)$ is the projection of $\calZ$ onto the subspace spanned by the first $\ell$ columns of $Q$; its volume equals $\mathrm{Vol}_\ell([q_1|\cdots|q_\ell]^\top\calZ)$ and does not depend on the choice of $Q$. We will only consider the case when the convex body $\calZ$ is defined as the convex hull of a finite set of points. In this case, the projection of the convex hull is the convex hull of the projection of the finite set of points defining $\calZ$.

The goal is to compute $T_\ell(\calZ)$ for all $\ell=1,\dots,p$. We restrict our attention to small $p$ where the volumes $\mathrm{Vol}_\ell(\mathrm{Proj}_{[q_1 | \cdots | q_\ell]}(\calZ))$ with $\ell \leq p$ may be efficiently computed using the quickhull algorithm \citep{barber1996quickhull} and its wrapper in \texttt{scipy.\allowbreak spatial.\allowbreak ConvexHull} \citep{2020SciPy-NMeth} for $\ell\ge2$, while for $\ell=1$ the volume is simply the range of the projected points. Rather than generating different Grassmann point sets $X_1,\dots,X_N$ and weights $w_1,\dots,w_N$ for each $T_\ell(\calZ)$ approximation, we generate a single set of points $X_1,\dots,X_N$ and weights $w_1,\dots,w_N$ on $\Flag(1,2,\dots,p,\bbR^n)$ with $\Delta_a = D/\lVert D\rVert_\mathrm{F}$ where $D = \diag(p+1,p,\dots,2,I_{n-p})$, i.e., $a_j = (p+2-j)/\lVert D\rVert_\mathrm{F}$ for $1\le j\le p+1$ with $\lVert D\rVert_\mathrm{F}^2 = \sum_{j=2}^{p+1}j^2+(n-p)$ (so $\Delta_a=\diag(3,2,1)/\sqrt{14}$ and $\diag(3,2,1,1,1)/\sqrt{17}$ in our experiments); since $\spn(q_1,\dots,q_\ell)$ is determined by $X=Q\Delta_aQ^\top$ and is uniform on $\Gr(\ell,\bbR^n)$
 under $\mu$, these may be used to simultaneously approximate $(T_1(\calZ),\dots,T_p(\calZ))$. 

We test our optimal cubature methods by simultaneously approximating the first $p$ intrinsic volumes of the hypercube $[-1,1]^n$ where we know $T_\ell([-1,1]^n) = \binom{n}{\ell} 2^\ell$. In \Cref{fig:Vk.cube}, the top row considers a smaller scale case where $n=3$ with $p=2$ and $\nu=1/2$ so $d=3$ and $s=2$, giving the optimal rate $\calO(N^{-s/d}) = \calO(N^{-2/3})$, while the bottom row considers a larger scale case where $n=5$ with $p=2$ and $\nu=3/2$ so $d=7$ and $s=5$, giving the optimal rate $\calO(N^{-s/d}) = \calO(N^{-5/7})$. For both cases we get the anticipated result: independent and identically distributed points and equal weights obtain error convergence rates like $\calO(N^{-1/2})$, while independent and identically distributed points with optimized weights obtain error convergence rates like $\calO(N^{-s/d})$. Notice that the approximation error may not lie below the worst-case error as the worst-case error is over functions with unit norm in the reproducing kernel Hilbert space, which is not necessarily the case for the integrands here. 

\section{Conclusion} \label{sec:conclusion} 

This article developed optimal order (up to logarithmic factors) weighted cubature rules for Monte Carlo numerical integration in Sobolev reproducing kernel Hilbert spaces on Stiefel, Grassmann, flag, and Lagrangian Grassmann matrix manifolds over $\bbR$, $\bbC$, and $\bbH$. The proposed method for generating cubature points and weights relies on two critical innovations: (1) the use of Mat\'ern kernels of Frobenius (chordal) distances, which we showed exactly reproduce the Sobolev spaces of interest with equivalent norms, and (2) the use of worst-case-error-minimizing  weights constrained to sum to one, which circumvented the requirement to compute difficult-to-evaluate kernel mean embeddings (kernel integrals). These components enabled a simple implementation of our method, for which we provided the open-source software \texttt{rmatqmcpy} (\url{https://github.com/alegresor/rmatqmcpy}). Beyond numerical validation of optimal order worst-case error rates, we demonstrated the utility of our method through an application to simultaneous approximation of intrinsic volumes in the Cauchy--Kubota formula.

\section*{Acknowledgments}

We would like to thank Chris Peterson for discussing applications to intrinsic volume computations. This work is partially supported by DARPA grants HR0011-21-9-0040 and HR0011-25-9-0031.

\bibliographystyle{amsplain}
\bibliography{main}

@article{kubota1925konvex,
  title               = {{\"U}ber die konvex-geschlossenen {M}annigfaltigkeiten im $n$-dimensionalen {R}aume},
  author              = {Kubota, Tadahiko},
  year                = {1925},
  journal             = {Science Reports of the T{\^o}hoku Imperial University. Series 1},
  volume              = {14},
  pages               = {85--99},
}

@article{levenberg1944method,
  title               = {A method for the solution of certain non-linear problems in least squares},
  author              = {Levenberg, Kenneth},
  year                = {1944},
  journal             = {Quart. Appl. Math.},
  volume              = {2},
  pages               = {164--168},
  doi                 = {10.1090/qam/10666},
  issn                = {0033-569X,1552-4485},
  url                 = {https://doi.org/10.1090/qam/10666},
  fjournal            = {Quarterly of Applied Mathematics},
  mrclass             = {65.0X},
  mrnumber            = {10666},
  mrreviewer          = {T.\ E.\ Sterne},
}

@book{Hadwiger1957,
  title               = {Vorlesungen {\"u}ber {I}nhalt, {O}berfl{\"a}che und {I}soperimetrie},
  author              = {Hadwiger, Hugo},
  year                = {1957},
  publisher           = {Springer-Verlag},
  address             = {Berlin, Heidelberg},
  series              = {Die Grundlehren der mathematischen Wissenschaften in Einzeldarstellungen},
  doi                 = {10.1007/978-3-642-94702-5},
}

@article{marquardt1963algorithm,
  title               = {An algorithm for least-squares estimation of nonlinear parameters},
  author              = {Marquardt, Donald W.},
  year                = {1963},
  journal             = {J. Soc. Indust. Appl. Math.},
  volume              = {11},
  pages               = {431--441},
  issn                = {0368-4245},
  fjournal            = {Journal of the Society for Industrial and Applied Mathematics},
  mrclass             = {62.20},
  mrnumber            = {153071},
  mrreviewer          = {M.\ Atiqullah},
}

@article{james1964distributions,
  title               = {Distributions of matrix variates and latent roots derived from normal samples},
  author              = {James, Alan T.},
  year                = {1964},
  journal             = {Ann. Math. Statist.},
  volume              = {35},
  pages               = {475--501},
  doi                 = {10.1214/aoms/1177703550},
  issn                = {0003-4851},
  url                 = {https://doi.org/10.1214/aoms/1177703550},
  fjournal            = {Annals of Mathematical Statistics},
  mrclass             = {62.40},
  mrnumber            = {181057},
  mrreviewer          = {I.\ Olkin},
}

@incollection{more1978levenberg,
  title               = {The {L}evenberg-{M}arquardt algorithm: implementation and theory},
  author              = {Mor\'e, Jorge J.},
  year                = {1978},
  booktitle           = {Numerical analysis ({P}roc. 7th {B}iennial {C}onf., {U}niv. {D}undee, {D}undee, 1977)},
  publisher           = {Springer, Berlin-New York},
  series              = {Lecture Notes in Math.},
  volume              = {630},
  pages               = {105--116},
  isbn                = {3-540-08538-6},
  mrclass             = {65K05},
  mrnumber            = {483445},
}

@book{muirhead1982aspects,
  title               = {Aspects of multivariate statistical theory},
  author              = {Muirhead, Robb J.},
  year                = {1982},
  publisher           = {John Wiley \& Sons, Inc., New York},
  series              = {Wiley Series in Probability and Mathematical Statistics},
  pages               = {xix+673},
  isbn                = {0-471-09442-0},
  mrclass             = {62Hxx},
  mrnumber            = {652932},
  mrreviewer          = {Nariaki\ Sugiura},
}

@book{triebel.theory_function_spaces_2,
  title               = {Theory of function spaces. {II}},
  author              = {Triebel, Hans},
  year                = {1992},
  publisher           = {Birkh\"auser Verlag, Basel},
  series              = {Monographs in Mathematics},
  volume              = {84},
  pages               = {viii+370},
  doi                 = {10.1007/978-3-0346-0419-2},
  isbn                = {3-7643-2639-5},
  url                 = {https://doi.org/10.1007/978-3-0346-0419-2},
  mrclass             = {46Exx (46-02)},
  mrnumber            = {1163193},
  mrreviewer          = {P.\ Szeptycki},
}

@article{barber1996quickhull,
  title               = {The quickhull algorithm for convex hulls},
  author              = {Barber, C. Bradford and Dobkin, David P. and Huhdanpaa, Hannu},
  year                = {1996},
  journal             = {ACM Trans. Math. Software},
  volume              = {22},
  number              = {4},
  pages               = {469--483},
  doi                 = {10.1145/235815.235821},
  issn                = {0098-3500,1557-7295},
  url                 = {https://doi.org/10.1145/235815.235821},
  fjournal            = {Association for Computing Machinery. Transactions on Mathematical Software},
  mrclass             = {65Y25 (65D17)},
  mrnumber            = {1428265},
}

@book{klain.intro_geometric_prob,
  title               = {Introduction to geometric probability},
  author              = {Klain, Daniel A. and Rota, Gian-Carlo},
  year                = {1997},
  publisher           = {Cambridge University Press, Cambridge},
  series              = {Lezioni Lincee. [Lincei Lectures]},
  pages               = {xiv+178},
  isbn                = {0-521-59362-X; 0-521-59654-8},
  mrclass             = {52A22 (52A35 52A40 52A55 52B45)},
  mrnumber            = {1608265},
  mrreviewer          = {Daniel\ M.\ Burns, Jr.},
}

@article{hickernell.generalized_discrepancy_quadrature_error_bound,
  title               = {A generalized discrepancy and quadrature error bound},
  author              = {Hickernell, Fred J.},
  year                = {1998},
  journal             = {Math. Comp.},
  volume              = {67},
  number              = {221},
  pages               = {299--322},
  doi                 = {10.1090/S0025-5718-98-00894-1},
  issn                = {0025-5718,1088-6842},
  url                 = {https://doi.org/10.1090/S0025-5718-98-00894-1},
  fjournal            = {Mathematics of Computation},
  mrclass             = {65D32},
  mrnumber            = {1433265},
  mrreviewer          = {J.\ Kofro\v n},
}

@article{edelman1998geometry,
  title               = {The geometry of algorithms with orthogonality constraints},
  author              = {Edelman, Alan and Arias, Tom\'as A. and Smith, Steven T.},
  year                = {1999},
  journal             = {SIAM J. Matrix Anal. Appl.},
  volume              = {20},
  number              = {2},
  pages               = {303--353},
  doi                 = {10.1137/S0895479895290954},
  issn                = {0895-4798,1095-7162},
  url                 = {https://doi.org/10.1137/S0895479895290954},
  fjournal            = {SIAM Journal on Matrix Analysis and Applications},
  mrclass             = {65F30 (15A18 51F20 65H10 90C30)},
  mrnumber            = {1646856},
  mrreviewer          = {Michael\ M.\ Dediu},
}

@article{guionnet2002large,
  title               = {Large deviations and stochastic calculus for large random matrices},
  author              = {Guionnet, Alice},
  year                = {2004},
  journal             = {Probab. Surv.},
  volume              = {1},
  pages               = {72--172},
  doi                 = {10.1214/154957804100000033},
  issn                = {1549-5787},
  url                 = {https://doi.org/10.1214/154957804100000033},
  fjournal            = {Probability Surveys},
  mrclass             = {60F10 (15A52 46L54 60J65 82B31)},
  mrnumber            = {2095566},
  mrreviewer          = {Dominique\ L\'epingle},
}

@book{bengtsson2006geometry,
  title               = {Geometry of quantum states},
  author              = {Bengtsson, Ingemar and \.Zyczkowski, Karol},
  year                = {2006},
  publisher           = {Cambridge University Press, Cambridge},
  pages               = {xii+466},
  doi                 = {10.1017/CBO9780511535048},
  isbn                = {978-0-521-81451-5; 0-521-81451-0},
  url                 = {https://doi.org/10.1017/CBO9780511535048},
  note                = {An introduction to quantum entanglement},
  mrclass             = {81-01 (47N50 53C80 58B20 81P15 81P68)},
  mrnumber            = {2230995},
  mrreviewer          = {Paul\ B.\ Slater},
}

@book{rasmussen.gp4ml,
  title               = {Gaussian processes for machine learning},
  author              = {Rasmussen, Carl Edward and Williams, Christopher K. I.},
  year                = {2006},
  publisher           = {MIT Press, Cambridge, MA},
  series              = {Adaptive Computation and Machine Learning},
  pages               = {xviii+248},
  isbn                = {978-0-262-18253-9},
  mrclass             = {68T05 (60G15 62G08 62H30 93E35)},
  mrnumber            = {2514435},
  mrreviewer          = {Wenbo\ V.\ Li},
}

@book{SchneiderWeil2008,
  title               = {Stochastic and integral geometry},
  author              = {Schneider, Rolf and Weil, Wolfgang},
  year                = {2008},
  publisher           = {Springer-Verlag, Berlin},
  series              = {Probability and its Applications (New York)},
  pages               = {xii+693},
  doi                 = {10.1007/978-3-540-78859-1},
  isbn                = {978-3-540-78858-4},
  url                 = {https://doi.org/10.1007/978-3-540-78859-1},
  mrclass             = {60-02 (52A22 60D05 60G55 62M30)},
  mrnumber            = {2455326},
  mrreviewer          = {V.\ K.\ Ohanyan},
}

@article{plaskota2009new,
  title               = {New averaging technique for approximating weighted integrals},
  author              = {Plaskota, L. and Wasilkowski, G. W. and Zhao, Y.},
  year                = {2009},
  journal             = {J. Complexity},
  volume              = {25},
  number              = {3},
  pages               = {268--291},
  doi                 = {10.1016/j.jco.2009.02.004},
  issn                = {0885-064X,1090-2708},
  url                 = {https://doi.org/10.1016/j.jco.2009.02.004},
  fjournal            = {Journal of Complexity},
  mrclass             = {65D30},
  mrnumber            = {2524547},
  mrreviewer          = {Aur\'el\ Gal\'antai},
}

@book{dick2010digital,
  title               = {Digital nets and sequences},
  author              = {Dick, Josef and Pillichshammer, Friedrich},
  year                = {2010},
  publisher           = {Cambridge University Press, Cambridge},
  pages               = {xviii+600},
  doi                 = {10.1017/CBO9780511761188},
  isbn                = {978-0-521-19159-3},
  url                 = {https://doi.org/10.1017/CBO9780511761188},
  note                = {Discrepancy theory and quasi-{M}onte {C}arlo integration},
  mrclass             = {65C05 (11K06 11K38 11K45 46E35 65D30)},
  mrnumber            = {2683394},
  mrreviewer          = {Petru\ P.\ Blaga},
}

@article{hinrichs2010optimal,
  title               = {Optimal importance sampling for the approximation of integrals},
  author              = {Hinrichs, Aicke},
  year                = {2010},
  journal             = {J. Complexity},
  volume              = {26},
  number              = {2},
  pages               = {125--134},
  doi                 = {10.1016/j.jco.2009.11.003},
  issn                = {0885-064X,1090-2708},
  url                 = {https://doi.org/10.1016/j.jco.2009.11.003},
  fjournal            = {Journal of Complexity},
  mrclass             = {65D30 (46E22 46N40 47B10 47G10 65C05 68W25)},
  mrnumber            = {2607728},
  mrreviewer          = {Joanna\ E.\ Kapusta},
}

@book{novak2010tractability,
  title               = {Tractability of multivariate problems. {V}olume {II}: {S}tandard information for functionals},
  author              = {Novak, Erich and Wo\'zniakowski, Henryk},
  year                = {2010},
  publisher           = {European Mathematical Society (EMS), Z\"urich},
  series              = {EMS Tracts in Mathematics},
  volume              = {12},
  pages               = {xviii+657},
  doi                 = {10.4171/084},
  isbn                = {978-3-03719-084-5},
  url                 = {https://doi.org/10.4171/084},
  mrclass             = {46N40 (41A55 41A63 47A50 47N40 65Y20 68Q25)},
  mrnumber            = {2676032},
  mrreviewer          = {Arthur\ G.\ Werschulz},
}

@article{brauchart2012quasi,
  title               = {Quasi-{M}onte {C}arlo rules for numerical integration over the unit sphere {$\mathbb{S}^2$}},
  author              = {Brauchart, Johann S. and Dick, Josef},
  year                = {2012},
  journal             = {Numer. Math.},
  volume              = {121},
  number              = {3},
  pages               = {473--502},
  doi                 = {10.1007/s00211-011-0444-6},
  issn                = {0029-599X,0945-3245},
  url                 = {https://doi.org/10.1007/s00211-011-0444-6},
  fjournal            = {Numerische Mathematik},
  mrclass             = {65C05 (65D30)},
  mrnumber            = {2929076},
  mrreviewer          = {Borek\ D.\ Puza},
}

@article{FuselierWright2012,
  title               = {Scattered data interpolation on embedded submanifolds with restricted positive definite kernels: {S}obolev error estimates},
  author              = {Fuselier, Edward and Wright, Grady B.},
  year                = {2012},
  journal             = {SIAM J. Numer. Anal.},
  volume              = {50},
  number              = {3},
  pages               = {1753--1776},
  doi                 = {10.1137/110821846},
  issn                = {0036-1429,1095-7170},
  url                 = {https://doi.org/10.1137/110821846},
  fjournal            = {SIAM Journal on Numerical Analysis},
  mrclass             = {41A05 (41A30 65D05)},
  mrnumber            = {2970764},
  mrreviewer          = {Jeremy\ Levesley},
}

@article{dick2013high,
  title               = {High-dimensional integration: the quasi-{M}onte {C}arlo way},
  author              = {Dick, Josef and Kuo, Frances Y. and Sloan, Ian H.},
  year                = {2013},
  journal             = {Acta Numer.},
  volume              = {22},
  pages               = {133--288},
  doi                 = {10.1017/S0962492913000044},
  issn                = {0962-4929,1474-0508},
  url                 = {https://doi.org/10.1017/S0962492913000044},
  fjournal            = {Acta Numerica},
  mrclass             = {65D30 (65C05)},
  mrnumber            = {3038697},
  mrreviewer          = {Daniela\ Ro\c sca},
}

@phdthesis{graef2013efficient,
  title               = {Efficient Algorithms for the Computation of Optimal Quadrature Points on {R}iemannian Manifolds},
  author              = {Gr{\"a}f, Manuel},
  year                = {2013},
  isbn                = {978-3-941003-89-7},
  url                 = {https://nbn-resolving.org/urn:nbn:de:bsz:ch1-qucosa-115287},
  school              = {TU Chemnitz},
}

@article{brandolini2010quadrature,
  title               = {Quadrature rules and distribution of points on manifolds},
  author              = {Brandolini, Luca and Choirat, Christine and Colzani, Leonardo and Gigante, Giacomo and Seri, Raffaello and Travaglini, Giancarlo},
  year                = {2014},
  journal             = {Ann. Sc. Norm. Super. Pisa Cl. Sci. (5)},
  volume              = {13},
  number              = {4},
  pages               = {889--923},
  issn                = {0391-173X,2036-2145},
  fjournal            = {Annali della Scuola Normale Superiore di Pisa. Classe di Scienze. Serie V},
  mrclass             = {41A55 (11K38 42C10 65D30)},
  mrnumber            = {3362114},
  mrreviewer          = {S.\ Hitotumatu},
}

@article{brauchart2014qmc,
  title               = {Q{MC} designs: optimal order quasi {M}onte {C}arlo integration schemes on the sphere},
  author              = {Brauchart, J. S. and Saff, E. B. and Sloan, I. H. and Womersley, R. S.},
  year                = {2014},
  journal             = {Math. Comp.},
  volume              = {83},
  number              = {290},
  pages               = {2821--2851},
  doi                 = {10.1090/S0025-5718-2014-02839-1},
  issn                = {0025-5718,1088-6842},
  url                 = {https://doi.org/10.1090/S0025-5718-2014-02839-1},
  fjournal            = {Mathematics of Computation},
  mrclass             = {65D30 (11K38 41A55 65C05)},
  mrnumber            = {3246811},
  mrreviewer          = {S.\ Hitotumatu},
}

@article{OszmaniecAugusiakGogolinKolodynskiAcinLewenstein2016,
  title               = {Random Bosonic States for Robust Quantum Metrology},
  author              = {Oszmaniec, M. and Augusiak, R. and Gogolin, C. and Ko\l{}ody\ifmmode \acute{n}\else \'{n}\fi{}ski, J. and Ac\'{\i}n, A. and Lewenstein, M.},
  year                = {2016},
  month               = dec,
  journal             = {Phys. Rev. X},
  publisher           = {American Physical Society},
  volume              = {6},
  pages               = {041044},
  doi                 = {10.1103/PhysRevX.6.041044},
  url                 = {https://link.aps.org/doi/10.1103/PhysRevX.6.041044},
  issue               = {4},
  numpages            = {34},
}

@incollection{breger2017quasi,
  title               = {Quasi {M}onte {C}arlo integration and kernel-based function approximation on {G}rassmannians},
  author              = {Breger, Anna and Ehler, Martin and Gr\"af, Manuel},
  year                = {2017},
  booktitle           = {Frames and other bases in abstract and function spaces},
  publisher           = {Birkh\"auser/Springer, Cham},
  series              = {Appl. Numer. Harmon. Anal.},
  pages               = {333--353},
  isbn                = {978-3-319-55549-2; 978-3-319-55550-8},
  mrclass             = {65C05 (65D30)},
  mrnumber            = {3700125},
  mrreviewer          = {Takashi\ Goda},
}

@article{breger2018points,
  title               = {Points on manifolds with asymptotically optimal covering radius},
  author              = {Breger, Anna and Ehler, Martin and Gr\"af, Manuel},
  year                = {2018},
  journal             = {J. Complexity},
  volume              = {48},
  pages               = {1--14},
  doi                 = {10.1016/j.jco.2018.03.002},
  issn                = {0885-064X,1090-2708},
  url                 = {https://doi.org/10.1016/j.jco.2018.03.002},
  fjournal            = {Journal of Complexity},
  mrclass             = {65D30 (65C05 65Y20)},
  mrnumber            = {3828835},
  mrreviewer          = {Takashi\ Goda},
}

@inproceedings{karvonen2018bayes,
  title               = {A {B}ayes--{S}ard Cubature Method},
  author              = {Karvonen, Toni and Oates, Chris J. and S{\"a}rkk{\"a}, Simo},
  year                = {2018},
  booktitle           = {Advances in Neural Information Processing Systems},
  publisher           = {Curran Associates, Inc.},
  volume              = {31},
  url                 = {https://proceedings.neurips.cc/paper\%5Ffiles/paper/2018/file/6775a0635c302542da2c32aa19d86be0-Paper.pdf},
  editor              = {S. Bengio and H. Wallach and H. Larochelle and K. Grauman and N. Cesa-Bianchi and R. Garnett},
}

@article{briol2015probabilistic,
  title               = {Probabilistic integration: a role in statistical computation?},
  author              = {Briol, Fran\c cois-Xavier and Oates, Chris J. and Girolami, Mark and Osborne, Michael A. and Sejdinovic, Dino},
  year                = {2019},
  journal             = {Statist. Sci.},
  volume              = {34},
  number              = {1},
  pages               = {1--22},
  doi                 = {10.1214/18-STS660},
  issn                = {0883-4237,2168-8745},
  url                 = {https://doi.org/10.1214/18-STS660},
  fjournal            = {Statistical Science. A Review Journal of the Institute of Mathematical Statistics},
  mrclass             = {62-02 (60-08 60B11 60G15 62F15 62G30 65C05)},
  mrnumber            = {3938958},
}

@article{ehler2019optimal,
  title               = {Optimal {M}onte {C}arlo integration on closed manifolds},
  author              = {Ehler, Martin and Gr\"af, Manuel and Oates, Chris J.},
  year                = {2019},
  journal             = {Stat. Comput.},
  volume              = {29},
  number              = {6},
  pages               = {1203--1214},
  doi                 = {10.1007/s11222-019-09894-w},
  issn                = {0960-3174,1573-1375},
  url                 = {https://doi.org/10.1007/s11222-019-09894-w},
  fjournal            = {Statistics and Computing},
  mrclass             = {65C05 (62F15)},
  mrnumber            = {4026666},
  mrreviewer          = {Adam\ M.\ Johansen},
}

@article{2020SciPy-NMeth,
  title               = {{{SciPy} 1.0: Fundamental Algorithms for Scientific Computing in {P}ython}},
  author              = {Virtanen, Pauli and Gommers, Ralf and Oliphant, Travis E. and Haberland, Matt and Reddy, Tyler and Cournapeau, David and Burovski, Evgeni and Peterson, Pearu and Weckesser, Warren and Bright, Jonathan and {van der Walt}, St{\'e}fan J. and Brett, Matthew and Wilson, Joshua and Millman, K. Jarrod and Mayorov, Nikolay and Nelson, Andrew R. J. and Jones, Eric and Kern, Robert and Larson, Eric and Carey, C J and Polat, {\.I}lhan and Feng, Yu and Moore, Eric W. and {VanderPlas}, Jake and Laxalde, Denis and Perktold, Josef and Cimrman, Robert and Henriksen, Ian and Quintero, E. A. and Harris, Charles R. and Archibald, Anne M. and Ribeiro, Ant{\^o}nio H. and Pedregosa, Fabian and {van Mulbregt}, Paul and {SciPy 1.0 Contributors}},
  year                = {2020},
  journal             = {Nat. Methods},
  volume              = {17},
  pages               = {261--272},
  doi                 = {10.1038/s41592-019-0686-2},
  adsurl              = {https://rdcu.be/b08Wh},
}

@inproceedings{borovitskiy2020matern,
  title               = {Mat\'{e}rn {G}aussian Processes on {R}iemannian Manifolds},
  author              = {Borovitskiy, Viacheslav and Terenin, Alexander and Mostowsky, Peter and Deisenroth, Marc},
  year                = {2020},
  booktitle           = {Advances in Neural Information Processing Systems},
  publisher           = {Curran Associates, Inc.},
  volume              = {33},
  pages               = {12426--12437},
  url                 = {https://proceedings.neurips.cc/paper\%5Ffiles/paper/2020/file/92bf5e6240737e0326ea59846a83e076-Paper.pdf},
  editor              = {H. Larochelle and M. Ranzato and R. Hadsell and M.F. Balcan and H. Lin},
}

@article{de2021reproducing,
  title               = {Reproducing kernel {H}ilbert spaces on manifolds: {S}obolev and diffusion spaces},
  author              = {De Vito, Ernesto and M\"ucke, Nicole and Rosasco, Lorenzo},
  year                = {2021},
  journal             = {Anal. Appl. (Singap.)},
  volume              = {19},
  number              = {3},
  pages               = {363--396},
  doi                 = {10.1142/S0219530520400114},
  issn                = {0219-5305,1793-6861},
  url                 = {https://doi.org/10.1142/S0219530520400114},
  fjournal            = {Analysis and Applications},
  mrclass             = {46E22 (46E35 58C25)},
  mrnumber            = {4233705},
  mrreviewer          = {William\ E.\ Gryc},
}

@misc{guo2025eigen,
  title               = {Eigen, singular, cosine-sine, and {A}utonne--{T}akagi vectors distributions of random matrix ensembles},
  author              = {Guo, Yihan and Lim, Lek-Heng},
  year                = {2025},
  howpublished        = {arXiv preprint arXiv:2512.12766},
  url                 = {https://arxiv.org/abs/2512.12766},
  eprint              = {2512.12766},
}

@article{LaiLimYe2025EDDManifoldOptimization,
  title               = {Euclidean distance degree in manifold optimization},
  author              = {Lai, Zehua and Lim, Lek-Heng and Ye, Ke},
  year                = {2025},
  journal             = {SIAM J. Optim.},
  volume              = {35},
  number              = {4},
  pages               = {2402--2422},
  doi                 = {10.1137/25M1735032},
  issn                = {1052-6234,1095-7189},
  url                 = {https://doi.org/10.1137/25M1735032},
  fjournal            = {SIAM Journal on Optimization},
  mrclass             = {14M15 (15A24 58C05 58K05 90C26)},
  mrnumber            = {4976079},
}

@article{LimYe2024MinimalEquivariant,
  title               = {Simple matrix expressions for the curvatures of {G}rassmannian},
  author              = {Lai, Zehua and Lim, Lek-Heng and Ye, Ke},
  year                = {2026},
  journal             = {Found. Comput. Math.},
  volume              = {26},
  pages               = {2053--2089},
  doi                 = {10.1007/s10208-025-09723-9},
  isbn                = {1615-3383},
  url                 = {https://doi.org/10.1007/s10208-025-09723-9},
  date                = {2025/08/27},
  id                  = {Lai2025},
}

@article{LimLuYe2025ProductsGrassmannians,
  title               = {Special {O}rthogonal, {S}pecial {U}nitary, and {S}ymplectic {G}roups as {P}roducts of {G}rassmannians},
  author              = {Lim, Lek-Heng and Lu, Xiang and Ye, Ke},
  year                = {2026},
  journal             = {SIAM J. Matrix Anal. Appl.},
  volume              = {47},
  number              = {2},
  pages               = {954--975},
  doi                 = {10.1137/25M1748305},
  issn                = {0895-4798,1095-7162},
  url                 = {https://doi.org/10.1137/25M1748305},
  fjournal            = {SIAM Journal on Matrix Analysis and Applications},
  mrclass             = {14M15 (15A23 22E10 22E15 51A50)},
  mrnumber            = {5090591},
}

\end{document}